\documentclass[11pt]{amsart}
\usepackage{mathrsfs}
\usepackage{amssymb}
\usepackage{verbatim}
\usepackage{hyperref}
\usepackage{color}
\usepackage{amsfonts}
\usepackage{mathrsfs}
\usepackage{amsmath}
\usepackage{amssymb}

\begin{document}
\newtheorem{theorem}{Theorem}
\newtheorem{proposition}[theorem]{Proposition}
\newtheorem{conjecture}[theorem]{Conjecture}
\def\theconjecture{\unskip}
\newtheorem{corollary}[theorem]{Corollary}
\newtheorem{lemma}[theorem]{Lemma}
\newtheorem{claim}[theorem]{Claim}
\newtheorem{sublemma}[theorem]{Sublemma}
\newtheorem{observation}[theorem]{Observation}
\theoremstyle{definition}
\newtheorem{definition}{Definition}
\newtheorem{notation}[definition]{Notation}
\newtheorem{remark}[definition]{Remark}
\newtheorem{question}[definition]{Question}
\newtheorem{questions}[definition]{Questions}
\newtheorem{example}[definition]{Example}
\newtheorem{problem}[definition]{Problem}
\newtheorem{exercise}[definition]{Exercise}

\numberwithin{theorem}{section}
\numberwithin{definition}{section}
\numberwithin{equation}{section}

\def\earrow{{\mathbf e}}
\def\rarrow{{\mathbf r}}
\def\uarrow{{\mathbf u}}
\def\varrow{{\mathbf V}}
\def\tpar{T_{\rm par}}
\def\apar{A_{\rm par}}

\def\reals{{\mathbb R}}
\def\torus{{\mathbb T}}
\def\heis{{\mathbb H}}
\def\integers{{\mathbb Z}}
\def\naturals{{\mathbb N}}
\def\complex{{\mathbb C}\/}
\def\distance{\operatorname{distance}\,}
\def\support{\operatorname{support}\,}
\def\dist{\operatorname{dist}\,}
\def\Span{\operatorname{span}\,}
\def\degree{\operatorname{degree}\,}
\def\kernel{\operatorname{kernel}\,}
\def\dim{\operatorname{dim}\,}
\def\codim{\operatorname{codim}}
\def\trace{\operatorname{trace\,}}
\def\Span{\operatorname{span}\,}
\def\dimension{\operatorname{dimension}\,}
\def\codimension{\operatorname{codimension}\,}
\def\nullspace{\scriptk}
\def\kernel{\operatorname{Ker}}
\def\ZZ{ {\mathbb Z} }
\def\p{\partial}
\def\rp{{ ^{-1} }}
\def\Re{\operatorname{Re\,} }
\def\Im{\operatorname{Im\,} }
\def\ov{\overline}
\def\eps{\varepsilon}
\def\lt{L^2}
\def\diver{\operatorname{div}}
\def\curl{\operatorname{curl}}
\def\etta{\eta}
\newcommand{\norm}[1]{ \|  #1 \|}
\def\expect{\mathbb E}
\def\bull{$\bullet$\ }

\def\blue{\color{blue}}
\def\red{\color{red}}

\def\xone{x_1}
\def\xtwo{x_2}
\def\xq{x_2+x_1^2}
\newcommand{\abr}[1]{ \langle  #1 \rangle}

\newcommand{\Norm}[1]{ \left\|  #1 \right\| }
\newcommand{\set}[1]{ \left\{ #1 \right\} }
\def\one{\mathbf 1}
\def\whole{\mathbf V}
\newcommand{\modulo}[2]{[#1]_{#2}}
\def \essinf{\mathop{\rm essinf}}
\def\scriptf{{\mathcal F}}
\def\scriptg{{\mathcal G}}
\def\scriptm{{\mathcal M}}
\def\scriptb{{\mathcal B}}
\def\scriptc{{\mathcal C}}
\def\scriptt{{\mathcal T}}
\def\scripti{{\mathcal I}}
\def\scripte{{\mathcal E}}
\def\scriptv{{\mathcal V}}
\def\scriptw{{\mathcal W}}
\def\scriptu{{\mathcal U}}
\def\scriptS{{\mathcal S}}
\def\scripta{{\mathcal A}}
\def\scriptr{{\mathcal R}}
\def\scripto{{\mathcal O}}
\def\scripth{{\mathcal H}}
\def\scriptd{{\mathcal D}}
\def\scriptl{{\mathcal L}}
\def\scriptn{{\mathcal N}}
\def\scriptp{{\mathcal P}}
\def\scriptk{{\mathcal K}}
\def\frakv{{\mathfrak V}}
\def\C{\mathbb{C}}
\def\D{\mathcal{D}}
\def\R{\mathbb{R}}
\def\Rn{{\mathbb{R}^n}}
\def\Sn{{{S}^{n-1}}}
\def\M{\mathbb{M}}
\def\N{\mathbb{N}}
\def\Q{{\mathbb{Q}}}
\def\Z{\mathbb{Z}}
\def\F{\mathcal{F}}
\def\L{\mathcal{L}}
\def\S{\mathcal{S}}
\def\supp{\operatorname{supp}}
\def\dist{\operatorname{dist}}
\def\essi{\operatornamewithlimits{ess\,inf}}
\def\esss{\operatornamewithlimits{ess\,sup}}
\author{Mingming Cao}
\author{Qingying Xue}
\address{
         School of Mathematical Sciences \\
         Beijing Normal University \\
         Laboratory of Mathematics and Complex Systems \\
         Ministry of Education \\
         Beijing 100875 \\
         People's Republic of China}
\email{m.cao@mail.bnu.edu.cn}
\email{qyxue@bnu.edu.cn}

\thanks{The second author was supported partly by NSFC
(No. 11471041), the Fundamental Research Funds for the Central Universities (No. 2014KJJCA10) and NCET-13-0065. \\ \indent Corresponding
author: Qingying Xue\indent Email: qyxue@bnu.edu.cn}


\date{May 16, 2016.}
\title[Littlewood-Paley operators]{$L^p$ boundedness of non-homogeneous Littlewood-Paley $g^*_{\lambda,\mu}$-function with non-doubling measures}
\maketitle
\begin{abstract}
It is well-known that the $L^p$ boundedness and weak $(1,1)$ estiamte $(\lambda>2)$ of the classical Littlewood-Paley $g_{\lambda}^{*}$-function were first studied by Stein, and the
weak $(p,p)$ $(p>1)$ estimate was later given by Fefferman for $\lambda=2/p$. In this paper, we investigated the $L^p(\mu)$ boundedness of the non-homogeneous Littlewood-Paley
$g_{\lambda,\mu}^{*}$-function with non-convolution type kernels and a power bounded measure $\mu$:
$$
g_{\lambda,\mu}^*(f)(x)
= \bigg(\iint_{\R^{n+1}_{+}} \Big(\frac{t}{t + |x - y|}\Big)^{m \lambda} |\theta_t^\mu f(y)|^2
\frac{d\mu(y) dt}{t^{m+1}}\bigg)^{1/2},\ x \in \Rn,\ \lambda > 1,
$$
where $\theta_t^\mu f(y) = \int_{\Rn} s_t(y,z) f(z) d\mu(z)$, and $s_t$ is a non-convolution type kernel.
Based on a big piece prior boundedness, we first gave a sufficient condition for the $L^p(\mu)$ boundedness of $g_{\lambda,\mu}^*$.
This was done by means of the non-homogeneous good lambda method.
Then, using the methods of dyadic analysis, we demonstrated a big piece global $Tb$ theorem.
Finally, we obtaind a sufficient and necessary condition for $L^p(\mu)$ boundedness of $g_{\lambda,\mu}^*$-function.
It is worth noting that our testing conditions are weak $(1,1)$ type with respect to measures.
\end{abstract}
\section{Introduction}
Dyadic techniques play important roles in Harmonic analysis and other fields. They originated from martingale inequalities including dyadic maximal and square functions, which can be
considered as the particular Doob's maximal function and Burkholder's square function. In recent decades, dyadic analysis has attracted renewed attentions because of the study on Haar
shifts. In 2000,
Petermichl \cite{P2000} showed that the Hilbert transform can be expressed as an average of some dyadic shifts. Petermichl's shift operator relates the probabilistic-combinatorial
fields
with the analytic realms. In addition, it not only associates to discrete martingale transforms but also owns much simpler structure and stronger localization property, which lead to
delicate applications in dealing with certain combinatorial considerations and weighted norm estimates \cite{P2007}. In 2002, the authors \cite{PTV} showed that the 
result in \cite{P2000} was also true for Riesz transforms in $\Rn$. They simplified the previous method and utilized the product Haar system over cubes in $\Rn$. Almost
simultaneously, by making use of the similar ideas as in \cite{P2000}, Petermichl and Volberg \cite{PV} obtained a sharp weighted estimate of the Ahlfors-Beurling operator, where the
corresponding dyadic model operator was martingale transform. In 2008, Petermichl \cite{P2008} got the sharp weighted bound for the Riesz transforms. This was mainly done by using a
modified product Haar system adapted to weighted situation and n-dimensional analog of the dyadic shift operators,

Later, in 2012, Hyt\"{o}nen, who made full use of the advantages of dyadic techniques, demonstrated the dyadic representation theorem \cite{H2012} and solved the celebrated $A_2$
conjecture. It provides a direct link between classical and dyadic analysis by showing that any Calder\'{o}n-Zygmund singular integral operator has a representation in terms of
certain simpler dyadic shift operators. Hyt\"{o}nen's representation theorem mainly concerns discretizing objects and therefore it reduces problems into a parallel dyadic world where
objects, statements and analysis are often easier. The dyadic approach has also been utilized to investigate $L^p$ boundedness of some operators. Essentially, in the study of $L^p$
boundedness using dyadic approach, the most important step is to seek simpler dyadic model operator to approximate the original operator. Although one hopes that the original results
can be recovered from the dyadic model operator, the transition was not always automatic. This requires the dyadic model operator to be proper enough.
Another remarkable work of Hyt\"{o}nen \cite{H2012} is that he improved the probabilistic methods and random dyadic grids, which were introduced by Nazarov, Treil and Volberg when
they studyied the Calder\'{o}n-Zygmund singular integrals on non-homogeneous spaces \cite{NTV2003}. Later on, many nice works have been done in succession by usign the dyadic
techniques and probabilistic methods, see \cite{HM2014}, \cite{LL}, \cite{M2012}, \cite{M2014}, \cite{MM-S} and \cite{Ou} for more details.

 On one hand, it was well known that the local Tb theorem for the standard Calder\'{o}n-Zygmund operator $T$ was first studied by Christ \cite{C}. He showed that $T$ is bounded on
 $L^2(X,\mu)$, whenever $X$ is a homogeneous space and $\mu$ is a doubling measure. After that, many good results have been achieved. Among these achievements are the celebrated works of Nazarov,
 Treil and Volberg \cite{NTV2002}, Auscher, Hofmann, Muscalu, Tao and Thiele \cite{AHMTT}, as well as Hofmann \cite{Hofmann}. Recently, Hyt\"{o}nen and Martikainen \cite{HM2012} obtained a local Tb theorem for the general
 Calder\'{o}n-Zygmund operator in nonhomogeneous space. Shortly afterwards, a local Tb theorem for standard Calder\'{o}n-Zygmund operator in nonhomogeneous space with
 non-scale-invariant $L^2$ testing condition was proved by Lacey and Martikainen \cite{LM-CZ}. They made full use of the techniques of nonhomogeneous and two-weight dyadic analysis.

On the other hand, as for the square functions, the nonhomogeneous local Tb theorems for Littlewood-Paley $g$-function were in turn presented in \cite{MMO}, \cite{LM-S} and
\cite{MM-Lq}. However, the testing conditions range from scale-invariant $L^\infty$ type to non-scale-invariant $L^2$ type, and then to non-scale-invariant $L^p(1<p<2)$ type. In the
latter case, the twisted martingale difference operators associated with stopping cubes were introduced to overcome the problems brought by the general testing functions and upper
doubling measures. Quite recently, Martikainen, Mourgoglou and Vuorinen \cite{MMV} studied the nonhomogeneous local Tb theorem for $g$-function with weak $(1,1)$ testing condition.
Not long ago, Cao, Li and Xue \cite{CLX} gave a characterization of two weight norm inequalities for the classical Littlewood-Paley $g_\lambda^*$ function of higher dimension. The
authors made use of the averaging identity over good Whitney regions, which made the proof quite brief. Still more recently, Cao and Xue \cite{CX} established a local $Tb$ theorem for
the non-homogeneous Littlewood-Paley $g_{\lambda}^{*}$-function with non-convolution type kernels and upper power bound measure $\mu$. Significantly, it was in local setting and with
$L^p$-type testing condition. The exponent $p \in (1,2)$ means that one had to search for some new approaches and more complicated techniques instead of the averaging identity over
good Whitney regions used in \cite{CLX}.

In this paper, taking into account the advantages of dyadic analysis, we are aiming to prove a certain local Tb theorem for the Littlewood-Paley $g_\lambda^*$ function in nonhomogeneous space by making use of dyadic techniques. We are particularly
interested in the weak $(1,1)$ testing condition. Before stating our main results, let us first introduce some definitions and notations.

Let $\mathfrak{M}(\Rn)$ be the space of all complex Borel measures in $\Rn$ equipped with the norm of total variation
$||\nu|| = |\nu|(\Rn)$. We will use, for minor convenience, $\ell^\infty$ metrics on $\Rn$ in this paper. Given a complex measure $\nu$, we define the Littlewood-Paley
$g_\lambda^*$-function as follows
\begin{align*}
g_{\lambda}^*(\nu)(x)
= \bigg(\iint_{\R^{n+1}_{+}} \Big(\frac{t}{t + |x - y|}\Big)^{m \lambda} |\theta_t \nu(y)|^2
\frac{d\mu(y) dt}{t^{m+1}}\bigg)^{1/2},\ x \in \Rn,\ \lambda > 1,
\end{align*}
where $\theta_t \nu(y) = \int_{\Rn} s_t(y,z) d\nu(z)$. In particular, we denote
\begin{align}\label{sf}
g_{\lambda,\mu}^*(f)(x)
= \bigg(\iint_{\R^{n+1}_{+}} \Big(\frac{t}{t + |x - y|}\Big)^{m \lambda} |\theta_t^\mu f(y)|^2
\frac{d\mu(y) dt}{t^{m+1}}\bigg)^{1/2},\ x \in \Rn,\ \lambda > 1,
\end{align}
where $\theta_t^\mu f(y) =\theta_t (f \mu)(y) = \int_{\Rn} s_t(y,z) f(z) d\mu(z)$, and the non-convolution kernel $s_t$ satisfies the following Standard conditions:

\vspace{0.3cm}
\noindent\textbf{Standard conditions.} The kernel $s_t : \R^{n} \times \R^{n} \rightarrow \C$ is assumed to satisfy the following estimates: for some
$\alpha>0$
\begin{enumerate}
\item [(1)] Size condition :
$$ |s_t(x,y)| \lesssim \frac{t^{\alpha}}{(t + |x - y|)^{m+\alpha}}.$$
\item [(2)] H\"{o}lder condition :
$$ |s_t(x,y) - s_t(x,y')| \lesssim \frac{|y - y'|^{\alpha}}{(t + |x - y|)^{m+\alpha}},\ \text{whenever}\  |y - y'| < t/2 ;$$
\end{enumerate}
Moreover, the non-homogeneous measure $\mu$ is a Borel measure on $\Rn$ satisfying the $power \ bounded$ condition: for some $m>0$,
$$
\mu(B(x,r)) \lesssim r^m, \ \quad  x \in \Rn, \ r>0.
$$
Now, we introduce the notation $g_{\lambda,Q}^*(\nu)$, the localized version of the operators $g_{\lambda}^*$ and $g_{\lambda,\mu}^*$ as follows: For a given cube $Q$,
$g_{\lambda,Q}^*(\nu)$ is defined by
\begin{align*}
g_{\lambda,Q}^*(\nu)(x)
= \bigg(\int_{0}^{\ell(Q)}\int_{\Rn} \Big(\frac{t}{t + |x - y|}\Big)^{m \lambda} |\theta_t \nu(y)|^2
\frac{d\mu(y) dt}{t^{m+1}}\bigg)^{1/2},\ \lambda > 1,
\end{align*}
and
$g_{\lambda,\mu,Q}^*(f)(x)=g_{\lambda,Q}^*(f \mu)(x)$.

Before stating our main results, we first need to give one more definition.
\begin{definition}
\begin{enumerate}
\item [(1)] Given $a,b>1$, a cube $Q \subset \Rn$ is called $(a,b)$-doubling for a given measure $\mu$ if $\mu(a Q) \leq b \mu(Q)$.
\item [(2)] Given $\mathfrak{C}>0$ we say that a cube $Q \subset \Rn$ has $\mathfrak{C}$-small boundary with respect to the measure $\mu$ if
$$
\mu \big(\{x \in 2Q; \dist(x,\partial Q) \leq \xi \ell(Q) \}\big) \leq \mathfrak{C} \xi \mu(2Q), \quad \hbox{\ \ for\  every\ } \xi > 0.
$$
\end{enumerate}
\end{definition}
Based on a big piece prior boundedness, we first obtain a sufficient condition for the $L^p(\mu)$ boundedness of
$g_{\lambda,\mu}^*$-function.
\begin{theorem}\label{Theorem-L^p}
Let $\lambda > 2$, $0 < \alpha \leq m(\lambda-2)$ and $\mu$ be a power bounded measure. Let $\beta > 0$ and $\mathfrak{C}$ be a big enough number, depending only on the dimension $n$,
and $\theta \in (0,1)$. Suppose that for each $(2,\beta)$-doubling cube $Q$ with $\mathfrak{C}$-small boundary, there exists a subset $G_Q \subset Q$ such that $\mu(G_Q) \geq \theta
\mu(Q)$ and
$g_{\lambda}^* : \mathfrak{M}(\Rn) \rightarrow L^{1,\infty}(\mu \lfloor G_Q)$ is bounded with a uniform constant independent of $Q$. Then $g_{\lambda,\mu}^*$ is bounded on $L^p(\mu)$
for any $p \in (1,\infty)$ with a constant depending on $p$ and the preceding constants.
\end{theorem}
Another main result is the following big piece global $Tb$ theorem for $g_{\lambda}^*$-function.
\begin{theorem}\label{Theorem-big}
Let $\lambda > 2$, $0 < \alpha \leq m(\lambda-2)$ and $Q \subset \Rn$ be a cube. Let $\sigma$ be a finite Borel measure in $\Rn$ with $\supp \sigma \subset Q$. Suppose that $b$ is a
function satisfying $||b||_{L^\infty(\sigma)} \leq C_b$. For every $w$, let $T_w$ be the union of the maximal dyadic cubes $R \in \D_w$ for which $|\langle b \rangle_{R}^\sigma| \leq
c_1$.
A measurable set $H \subset \Rn$ is assumed to enjoy the following properties:
\begin{enumerate}
\item [(a)] There exists a $\delta_0$, $0<\delta_0 < 1$, such that $\sigma(H \cup T_w) \leq \delta_0 \sigma(Q)$ for every $w$;
\item [(b)] Every ball $B_r$ of radius $r$ satisfying $\sigma(B_r) > C_0 r^m$ satisfies $B_r \subset H$;
\item [(c)]For some $s>0$, it holds that
$$
\sup_{\zeta > 0} \zeta^s \sigma \big(\{x \in Q \setminus H; g_{\lambda,\sigma,Q}^*(b)(x) > \zeta \}\big)
\leq c_2 \sigma(Q).
$$
\end{enumerate}
Then, there is a measurable set $G_Q$ satisfying $G_Q \subset Q \setminus H$ and the following properties:
\begin{enumerate}
\item [(i)] $\sigma(G_Q) \simeq \sigma(Q)$;
\item [(ii)] $|| \mathbf{1}_{G_Q} g_{\lambda,\sigma,Q}^*(f) ||_{L^2(\sigma)} \lesssim ||f||_{L^2(\sigma)}, \ \
             \text{\quad for every} \ f \in L^2(\sigma)$.
\end{enumerate}
\end{theorem}

Finally, we present the non-homogeneous local $Tb$ theorem with weak $(1,1)$ testing condition.
\begin{theorem}\label{Theorem-Local}
Let $\lambda > 2$, $0 < \alpha \leq m(\lambda-2)$ and $\mu$ be a power bounded measure and $B_1,B_2 < \infty$, $ \epsilon \in (0, 1)$ be given constants. Let
$\beta > 0$, and $\mathfrak{C}$ be large enough depending only on $n$. Suppose that for every $(2, \beta)$-doubling cube $Q \subset \Rn$ with $\mathfrak{C}$-small boundary, there
exists a complex measure $\nu_Q$, such that
\begin{enumerate}
\item [(1)] $\supp \nu_Q \subset Q$;
\item [(2)] $\mu(Q)=\nu_Q(Q)$;
\item [(3)] $||\nu_Q|| \leq B_1 \mu(Q)$;
\item [(4)] For all Borel sets $A \subset Q$ satisfying $\mu(A) \leq \epsilon \mu(Q)$, there holds
$|\nu_Q|(A) \leq \frac{||\nu_Q||}{32 B_1}$.
\end{enumerate}
Then, for any $p \in (1,\infty)$, $g_{\lambda,\mu}^* : L^p(\mu) \rightarrow L^p(\mu)$ if and only if
there exist $s > 0$ and a Borel set $U_Q \subset \Rn$ with
$|\nu_Q|(U_Q) \leq \frac{||\nu_Q||}{16 B_1}$
such that
$$
\sup_{\zeta >0} \zeta^s \mu \big(\{x \in Q \setminus U_Q; g_{\lambda,Q}^*(\nu_Q) > \zeta\}\big) \leq B_2 ||\nu_Q||,\hbox{\ for \ all \ }Q
\hbox{\  as \ above. }
$$
\end{theorem}
\begin{remark}If we take simply $d\mu(x)=dx$ and $s_t(y,z)=p_t(y-z)$, where $p$ is the classical Poisson kernel, then the above operator in (\ref{sf}) coincides with the $g_{\lambda}^*$-function defined and studied by Stein \cite{Stein1961} in 1961 and later studied by Fefferman \cite{Fe} in 1970, Muckenhoupt and Wheeden \cite{MR} in 1974, etc. It should be noted that even for this classical case, the results in this paper are also completely new.
\end{remark}
This paper contains some new ingredients as follows. First, by assuming that on
a big piece $g_{\lambda}^*$ is bounded from $\mathfrak{M}(\Rn)$ to
$L^{1,\infty}$, we obtain the $L^p(\mu)$ boundedness of $g_{\lambda,\mu}^*$ by using the non-homogeneous good lambda method. Then, we made use of modern concepts and probabilistic methods including random dyadic grid and good/bad cubes, which were quite effective tools to reduce our initial
estimates. As we see in \cite{CLX} and \cite{CX}, it is important to get the  reduction to good cubes. And so is it in this article. Additionally, the martingale difference operators
here were associated with transit cubes and some function $b$. Last but not least, we combined the dyadic techniques with classical methods. Compared with the classical methods, the
dyadic ones were more natural when splitting the summations or integral regions. In dyadic setting, we may boil corresponding analyses down into the inclusion or distance relationship
among dyadic cubes. And fortunately, the geometric structure of dyadic cubes are not so complex. Together with the vanishing property of martingale difference operators, the calculations in
dyadic case were much easier than that in classical case. Furthermore, in the case of $g$-function, one may use the maximal singular integral operators to control some particular terms. However, in the
case of $g_{\lambda}^*$-function, we need to use different approach to overcome this difficulty.

We organize this paper as follows:  In Section $\ref{Sec-lambda}$, we will give the proof of Theorem $\ref{Theorem-L^p}$. Section \ref{Sec-big} will be devoted to demonstrate the big piece $Tb$ theorem
$\ref{Theorem-big}$. In Section $\ref{Sec-local}$,
the proof of non-homogeneous local $Tb$ theorem $\ref{Theorem-Local}$ will be presented. Finally, in Section $\ref{Sec-M}$ and Section $\ref{Sec-RBMO}$, it is shown that the
$L^2(\mu)$ boundedness of $g_{\lambda,\mu}^*$ implies not only the $\mathfrak{M}(\Rn) \rightarrow L^{1,\infty}(\mu) $ boundedness, but also the $L^\infty(\mu)
\rightarrow RBMO(\mu)$ boundedness of $g_{\lambda,\mu}^*$.

\vspace{0.3cm}
\noindent\textbf{Acknowledgements.}
The first author would like to thank Henri Martikainen for some useful discussions which improved the quality of this paper.

\section{Non-homogeneous good lambda method}\label{Sec-lambda}
To prove Theorem $\ref{Theorem-L^p}$, we will utilize the non-homogeneous good lambda method given by Tolsa. However, the proof is nontrivial. The following
good lambda inequality provides a foundation for our analysis, and the proof will be postponed to the end of this section.
\begin{lemma}\label{good-lambda}Let the $t_0$-truncated version of $g_{\lambda,\mu}^*(f)$ be defined by
$$
g_{\lambda,\mu,t_0}^*(f)(x)
= \bigg(\int_{t_0}^{\infty} \int_{\Rn} \Big(\frac{t}{t + |x - y|}\Big)^{m \lambda} |\theta_t^\mu f(y)|^2
\frac{d\mu(y) dt}{t^{m+1}}\bigg)^{1/2},\ \ t_0 > 0.
$$
Then, for any $\epsilon > 0$, there exists $\delta=\delta(\epsilon)>0$ such that
$$
\mu \big(\big\{x;g_{\lambda,\mu,t_0}^*(f)(x) > (1+\epsilon)\xi, M_{\mu}f(x) \leq \delta \xi  \big\}\big)
\leq \Big( 1 - \frac{\theta}{16 \rho_0} \Big) \mu\big(\{x; g_{\lambda,\mu,t_0}^*(f)(x) > \xi\}\big),
$$
for any $\xi>0$ and every compactly supported and bounded $f \in L^p(\mu)$.
\end{lemma}

\noindent\textbf{Proof of Theorem $\ref{Theorem-L^p}$.}
It is enough to show that $g_{\lambda,\mu,t_0}^*$ is bounded on $L^p(\mu)$ uniformly in $t_0$.
Without loss of generality, we assume that $f \in L^p(\mu)$ is bounded and has a compact support. A priori assumption
$\big\|g_{\lambda,\mu,t_0}^* f\big\|_{L^p(\mu)} < \infty$ will be needed. The proof of it is contained in the proof of Claim $\ref{claim}$ below.

By Lemma $\ref{good-lambda}$, it follows that
\begin{equation}\aligned\label{theta/4}
&\mu \big(\{x;  g_{\lambda,\mu,t_0}^* (f)(x) > (1+\epsilon)\xi\}\big) \\
&\leq \mu \big(\{x;  g_{\lambda,\mu,t_0}^* (f)(x) > (1+\epsilon)\xi, M_{\mu}f(x) \leq \delta \xi \}\big)
 +\mu \big(\{x; M_{\mu}f(x) > \delta \xi \}\big) \\
&\leq \Big( 1 - \frac{\theta}{16 \rho_0} \Big) \mu\big(\{x; g_{\lambda,\mu,t_0}^*(f)(x) > \xi\}\big)
+\mu \big(\{x; M_{\mu}f(x) > \delta \xi \}\big).
\endaligned
\end{equation}
Note that
\begin{equation}\label{distribution}
\big\| f \big\|_{L^p(\mu)}^p = p \int_{0}^\infty \alpha^{p-1} \mu(\{x; |f(x)| > \alpha \}) d\alpha.
\end{equation}
Accordingly, by $(\ref{theta/4})$ and $(\ref{distribution})$, it yields that
\begin{align*}
\big\| g_{\lambda,\mu,t_0}^* (f) \big\|_{L^p(\mu)}^p
&=(1+\epsilon)^p p \int_{0}^\infty \xi^{p-1} \mu(\{x; g_{\lambda,\mu,t_0}^*(f)(x)| > (1+\epsilon)\xi \}) d\xi \\
&\leq (1+\epsilon)^p \Big( 1 - \frac{\theta}{16 \rho_0} \Big) p \int_{0}^\infty \xi^{p-1} \mu\big(\{x; g_{\lambda,\mu,t_0}^*f(x) > \xi\}\big) d\xi\\&\quad
+(1+\epsilon)^p p \int_{0}^\infty \xi^{p-1} \mu \big(\{x; M_{\mu}f(x) > \delta \xi \}\big) d\xi \\
&= (1+\epsilon)^p \Big( 1 - \frac{\theta}{16 \rho_0} \Big) \big\| g_{\lambda,\mu,t_0}^* (f) \big\|_{L^p(\mu)}^p + (1+\epsilon)^p \delta^{-p}
\big\| M_{\mu}f \big\|_{L^p(\mu)}^p.
\end{align*}
Taking $\epsilon>0$ such that
$$
(1+\epsilon)^p \Big( 1 - \frac{\theta}{16 \rho_0} \Big) = \Big( 1 - \frac{\theta}{32 \rho_0} \Big).
$$
Then the assumption $\big\|g_{\lambda,\mu,t_0}^* f\big\|_{L^p(\mu)} < \infty$ yields that
$$
\big\|g_{\lambda,\mu,t_0}^* (f)\big\|_{L^p(\mu)}
\lesssim_{\theta,\delta} \big\| M_{\mu}f \big\|_{L^p(\mu)}
\lesssim_{\theta,\delta} \big\| f \big\|_{L^p(\mu)}.
$$
This completes the proof of Theorem $\ref{Theorem-L^p}$.

\qed

The following estimates will be used at certain key points in the proof of Lemma $\ref{good-lambda}$ and in the other parts later.
\begin{lemma}\label{U(f)}
For any $x,x_0 \in \Rn$ and $t>0$, we have the pointwise control
\begin{equation}\label{U-V}
\mathcal{U}_t(f)(x):= \bigg(\int_{\Rn} \Big(\frac{t}{t + |x - y|}\Big)^{m \lambda} |\theta^{\mu}_t f(y)|^2 \frac{d\mu(y)}{t^m}\bigg)^{1/2}
\lesssim \mathcal{V}_t(f)(x),
\end{equation}
and
\begin{equation}\label{U-U-V}
\big| \mathcal{U}_t(f)(x) - \mathcal{U}_t(f)(x_0) \big|
\lesssim t^{-1} |x-x_0| \mathcal{V}_t(f)(x),
\end{equation}
where
$
\mathcal{V}_t(f)(x):= \int_{\Rn}\frac{t^\alpha}{(t+|x-z|)^{m+\alpha}}|f(z)|d\mu(z).
$
\end{lemma}
\begin{proof}
Using the size condition of the kernel, one may obtain that
$$
\mathcal{U}_t(f)(x)
\lesssim \bigg[\int_{\Rn}\Big(\frac{t}{t + |x - y|}\Big)^{m \lambda} \bigg(\int_{\Rn}\frac{t^\alpha}{(t+|y-z|)^{m+\alpha}}|f(z)|d\mu(z)\bigg)^2
\frac{d\mu(y)}{t^m}\bigg]^{\frac12}.
$$
Thus, it suffices to bound the following two parts:
$$
\mathcal{U}_{t,1}(f)(x,t)
:= \bigg[\int_{\Rn}\Big(\frac{t}{t + |x - y|}\Big)^{m \lambda} \bigg(\int_{|z-x| > 2|y-x|}\frac{t^\alpha|f(z)|}{(t+|y-z|)^{m+\alpha}}d\mu(z)\bigg)^2
\frac{d\mu(y)}{t^m}\bigg]^{\frac12},
$$
and
$$
\mathcal{U}_{t,2}(f)(x,t)
:= \bigg[\int_{\Rn}\Big(\frac{t}{t + |x - y|}\Big)^{m \lambda} \bigg(\int_{|z-x|\leq 2|y-x|}\frac{t^\alpha|f(z)|}{(t+|y-z|)^{m+\alpha}}d\mu(z)\bigg)^2 \frac{d\mu(y)}{t^m}\bigg]^{\frac12}.
$$
For $\mathcal{U}_{t,1}(f)$, it holds that $|y-z| \geq |x-z| - |x-y| \gtrsim |x-z|$. Hence, it yields that
$$
\mathcal{U}_{t,1}(f)(x,t)
\lesssim \mathcal{V}_t(f)(x) \bigg[ \int_{\Rn}\Big(\frac{t}{t + |x - y|}\Big)^{m \lambda} \frac{d\mu(y)}{t^m}\bigg]^{\frac12}
\lesssim \mathcal{V}_t(f)(x).
$$
For $\mathcal{U}_{t,12}(f)$, note that if $0 < \alpha \leq m(\lambda-2)/2$, one obtains
$$
\big(\frac{t}{t + |x - y|}\big)^{m \lambda/2}
\leq \big(\frac{t}{t + |x - y|}\big)^{m +\alpha}
\lesssim \frac{t^{m+\alpha}}{(t + |x - z|)^{m+\alpha}}.
$$
Combining with the Young inequality, it gives that
\begin{align*}
\mathcal{U}_{t,2}(f)(x,t)
&\lesssim \bigg[\int_{\Rn}\bigg(\int_{|z-x|\leq 2|y-x|}\frac{t^\alpha}{(t+|y-z|)^{m+\alpha}}\frac{t^{m+\alpha}|f(z)|}{(t + |x - z|)^{m+\alpha}}d\mu(z)\bigg)^2
\frac{d\mu(y)}{t^m}\bigg]^{\frac12} \\
&\lesssim t^{m/2} || \phi*\psi_x ||_{L^2(\mu)}
\leq t^{m/2} || \phi ||_{L^2(\mu)} || \psi_x ||_{L^1(\mu)}
\lesssim || \psi_x ||_{L^1(\mu)}
= \mathcal{V}_t(f)(x),
\end{align*}
where
$
\phi(z)= \frac{t^\alpha}{(t+|z|)^{m+\alpha}}$ and $\psi_x(z)= \frac{t^\alpha}{(t+|x-z|)^{m+\alpha}}|f(z)|.
$
This shows that $(\ref{U-V})$ is true. As for the inequality $(\ref{U-U-V})$, it is trivial to check the following fact.
$$
\mathcal{H}_{x,t}(y)
:=\big| \big(\frac{t}{t + |x - y|}\big)^{m \lambda/2} - \big(\frac{t}{t + |x_0 - y|}\big)^{m \lambda/2} \big|
\lesssim \frac{|x-x_0|}{t} \big(\frac{t}{t + |x - y|}\big)^{m \lambda/2} .
$$
Therefore, it follows from $(\ref{U-V})$ that
\begin{align*}
\big| \mathcal{U}_t(f)(x) - \mathcal{U}_t(f)(x_0) \big|
&\leq \bigg(\int_{\Rn} \mathcal{H}_{x,t}(y)^2 |\theta^{\mu}_t f(y)|^2 \frac{d\mu(y)}{t^m}\bigg)^{1/2} \\
&\lesssim t^{-1} |x-x_0| \mathcal{U}_t(f)(x)
\lesssim t^{-1} |x-x_0| \mathcal{V}_t(f)(x).
\end{align*}
\end{proof}

\begin{lemma}\label{T(f)}
Let $B$ be a ball or a cube, and $x,x' \in B$. Then there holds that
$$
\mathcal{T}(f)(x,x')
:= \big[\iint_{\R^{n+1}_+} \big| \mathcal{G}_{t,y}(\mathbf{1}_{\Rn \setminus 2 B} f)(x) -
\mathcal{G}_{t,y}(\mathbf{1}_{\Rn \setminus 2 B} f)(x')\big|^2 \frac{d\mu(y) dt}{t^{m+1}}\big]^{\frac12}
\lesssim M_{\mu}f(x),
$$
where
$
\mathcal{G}_{t,y}(f)(x):=\big(\frac{t}{t+|x-y|}\big)^{m \lambda/2}  \theta_t^\mu f(y).
$
\end{lemma}
\begin{proof}
First, we split the upper space $\R^{n+1}_+$ into four pieces as follows :
\begin{eqnarray*}
\mathfrak{D}_1 &:=& \big\{(y,t) \in \R^{n+1}_+;\ |y-x| \leq t, |y-x'| \leq t \big\} \\
\mathfrak{D}_2 &:=& \big\{(y,t) \in \R^{n+1}_+;\ |y-x| > t, |y-x'| \leq t \big\} \\
\mathfrak{D}_3 &:=& \big\{(y,t) \in \R^{n+1}_+;\ |y-x| \leq t, |y-x'| > t \big\} \\
\mathfrak{D}_4 &:=& \big\{(y,t) \in \R^{n+1}_+;\ |y-x| > t, |y-x'| > t \big\},
\end{eqnarray*}
Let $\mathcal{T}_i$ be the operator $\mathcal{T}$ with the integration domain limited to $\mathfrak{D}_i$. Then, we get
$$
\mathcal{T}(f)(x,x')
\leq \mathcal{T}_1(f)(x,x') + \mathcal{T}_2(f)(x,x') + \mathcal{T}_3(f)(x,x') + \mathcal{T}_4(f)(x,x').
$$
\vspace{0.3cm}
\noindent\textbf{$\bullet$ Estimates of $\mathcal{T}_1(f)$, $\mathcal{T}_2(f)$ and $\mathcal{T}_3(f)$.}
We may further decompose the domain $\mathfrak{D}_1$ by :
$$
\mathfrak{D}_{1,1}:= \mathfrak{D}_1 \cap \{(y,t);\ 0 < t \leq r(B) \},\quad
\mathfrak{D}_{1,2}:= \mathfrak{D}_1 \cap \{(y,t);\ t > r(B) \}.
$$
Hence, $\mathcal{T}_1(f) \leq \mathcal{T}_{1,1}(f) + \mathcal{T}_{1,2}(f)$.

We first consider the contribution of $\mathcal{T}_{1,1}(f)$.
From the summation condition and size condition, it follows that $t + |y-z| \geq |x-y| + |y-z| \geq |x-z|$ and
$$
\big| \mathcal{G}_{t,y}(\mathbf{1}_{\Rn \setminus 2 B} f)(x) -
\mathcal{G}_{t,y}(\mathbf{1}_{\Rn \setminus 2 B} f)(x')\big|
\lesssim |\theta_t^\mu(\mathbf{1}_{\Rn \setminus 2 B} f)(y)|
\lesssim \int_{\Rn \setminus 2B} \frac{t^\alpha|f(z)|}{(t+|y-z|)^{m+\alpha}}  d\mu(z).
$$
Thus, by the Minkowski inequality, we get
\begin{align*}
\mathcal{T}_{1,1}(f)(x,x')
&\lesssim \int_{\Rn \setminus 2B}|f(z)| \bigg(\int_{0}^{r(B)}\int_{\substack{|y-x| \leq t \\ |y-x'| \leq t}} \frac{t^{2\alpha}}{(t+|y-z|)^{2m+2\alpha}}\frac{d\mu(y)
dt}{t^{m+1}}\bigg)^{1/2}d\mu(z) \\
&\lesssim \int_{\Rn \setminus 2B} \frac{|f(z)|}{|x-z|^{m+\alpha}} \bigg(\int_{0}^{r(B)}\int_{|y-x| \leq t} t^{2\alpha-m} d\mu(y) \frac{dt}{t}\bigg)^{1/2}d\mu(z) \\
&\lesssim r(B)^\alpha \sum_{k=1}^\infty \int_{2^{k+1}B \setminus 2^k B}\frac{|f(z)|}{|x-z|^{m+\alpha}}d\mu(z)
\lesssim M_{\mu}f(x).
\end{align*}
As for $\mathcal{T}_{1,2}(f)$, note that $2(t+|y-z|) \geq t + |z-x| + t - |y-x| \geq t + |z-x|$ and
\begin{equation}\label{lambda-lambda}
\Big|\Big(\frac{t}{t+|x-y|}\Big)^{m \lambda/2} - \Big(\frac{t}{t+|x'-y|}\Big)^{m \lambda/2}\Big|
\lesssim r(B)\frac{t^{m\lambda/2}}{(t+|x-y|)^{m\lambda/2+1}}.
\end{equation}
Thereby, we obtain
\begin{align*}
\mathcal{T}_{1,2}(f)(x,x')
&\lesssim \int_{\Rn \setminus 2B}|f(z)| \bigg(\int_{r(B)}^\infty \int_{\substack{|y-x| \leq t \\ |y-x'| \leq t}}
\frac{r(B)^2}{t^2}\frac{t^{2\alpha}}{(t+|y-z|)^{2m+2\alpha}}\frac{d\mu(y) dt}{t^{m+1}}\bigg)^{1/2}d\mu(z) \\
&\lesssim \int_{\Rn \setminus 2B} \frac{|f(z)|}{|z-x|^{m+\alpha}} \bigg(\int_{0}^{r(B)}\int_{|y-x| \leq t} t^{2\alpha-m} d\mu(y) \frac{dt}{t}\bigg)^{1/2}d\mu(z)\lesssim M_{\mu}f(x).
\end{align*}

Since $\mathcal{T}_2(f)$ is symmetric with $\mathcal{T}_3(f)$, we only need to give the estimate of $\mathcal{T}_2(f)$. Fortunately, similarly argument still works as in the case of $\mathcal{T}_1(f)$. In fact, for
$\mathcal{T}_2(f)$, if $t \leq r(B)$, we may utilize
$t+|y-z| \geq |y-x'| + |y-z| \geq |z-x'|$ and the method that handles $\mathcal{T}_{1,1}(f)$ to obtain the desired estimate.
If $t > r(B)$, we can make use of the fact $|y-x| \simeq |y-x'|$ and the way that deals with $\mathcal{T}_{1,2}(f)$ to get the desired conclusion.

\vspace{0.3cm}
\noindent\textbf{$\bullet$ Estimate of $\mathcal{T}_4(f)$.}
Write
$$
\mathfrak{D}_{4,1}:= \mathfrak{D}_4 \cap \{(y,t);0 < t \leq r(B) \big\},\ \ \
\mathfrak{D}_{4,2}:= \mathfrak{D}_4 \cap \{(y,t);t > r(B) \}.
$$
From the inequality $(\ref{lambda-lambda})$, size condition and Minkowski's inequality, it follows that
\begin{align*}
\mathcal{T}_{4}(f)(x,x')
&\lesssim \int_{\Rn \setminus 2B}|f(z)| \bigg(\iint_{\mathfrak{D}_{4,1}} \frac{t^{m \lambda}}{(t+|x-y|)^{m\lambda}} \frac{t^{2\alpha}}{(t+|y-z|)^{2m+2\alpha}}\frac{d\mu(y)
dt}{t^{m+1}}\bigg)^{1/2}d\mu(z) \\
&\quad + \int_{\Rn \setminus 2B}|f(z)| \bigg(\iint_{\mathfrak{D}_{4,2}} \frac{r(B)^2 t^{m \lambda}}{(t+|x-y|)^{m\lambda+2}} \frac{t^{2\alpha}}{(t+|y-z|)^{2m+2\alpha}}\frac{d\mu(y)
dt}{t^{m+1}}\bigg)^{1/2}d\mu(z) \\
&:=\sum_{j=1}^2 \int_{\Rn \setminus 2B}|f(z)| \mathcal{J}_j(z)d\mu(z).
\end{align*}
To get $\mathcal{T}_{4}(f)(x,x') \lesssim M_{\mu}f(x)$, it is enough to show that
\begin{equation*}
\mathcal{J}_j(z) \lesssim \frac{r(B)^{\alpha}}{|z-x|^{m+\alpha}},\ \ \quad j=1,2.
\end{equation*}
Indeed, by the trivial estimate
\begin{equation}\label{y-z>t}
\int_{|y-z| > t} \frac{t^{m+2\alpha}}{(t+|y-z|)^{2m+2\alpha}} d\mu(y) \lesssim 1,
\end{equation}
it yields that
\begin{align*}
\mathcal{J}_1(z)^2
&\leq \int_{0}^{r(B)}\int_{\substack{|y-z| > t \\ |y-x| > |z-x|/2}} + \int_{0}^{r(B)}\int_{t < |y-x| \leq |z-x|/2} + \int_{0}^{r(B)}\int_{\substack{|y-z| \leq t \\ |y-x| > |z-x|/2}}
\\
&\lesssim \int_{0}^{r(B)}\frac{t^{2\alpha}}{|z-x|^{2m+2\alpha}} \int_{|y-z| > t} \frac{t^{m+2\alpha}}{(t+|x-y|)^{2m+2\alpha}} d\mu(y) \frac{dt}{t} \\
&\quad + \int_{0}^{r(B)}\frac{t^{2\alpha}}{|z-x|^{2m+2\alpha}} \int_{|y-x| > t}  \frac{t^{m \lambda-m}}{(t+|x-y|)^{m\lambda}}
d\mu(y) \frac{dt}{t} \\
&\quad + \int_{0}^{r(B)}\int_{|y-z| \leq t} \frac{t^{2\alpha-m}}{|z-x|^{2m+2\alpha}} d\mu(y) \frac{dt}{t}\lesssim \frac{r(B)^{2\alpha}}{|z-x|^{2m+2\alpha}},
\end{align*}
where in the last step, we have used the condition $m\lambda \geq 2m + 2\alpha$.

Finally, we need to analyze $\mathcal{J}_2(z)$. Denote the projection of $\mathfrak{D}_{4,2}$ on $\Rn$ by $\mathfrak{D}_{4,2}^{\bot}$, then
$$
\mathfrak{D}_{4,2}^{\bot}=\bigcup_{j=1}^4 \mathfrak{D}_{4,2}^{\bot,j},
$$
where
\begin{eqnarray*}
\mathfrak{D}_{4,2}^{\bot,1}&:=& \mathfrak{D}_{4,2}^{\bot} \bigcap \big\{y;|y-x| > 2 |z-x| \big\},\\
\mathfrak{D}_{4,2}^{\bot,2}&:=& \mathfrak{D}_{4,2}^{\bot} \bigcap \big\{y;|y-x| \leq 2 |z-x|,\ |y-z| > |z-x|/2 \big\},\\
\mathfrak{D}_{4,2}^{\bot,3}&:=& \mathfrak{D}_{4,2}^{\bot} \bigcap \big\{y;|y-x| \leq 2 |z-x|,\ |y-z| \leq |z-x|/2,\ |y-z|>t \big\},\\
\mathfrak{D}_{4,2}^{\bot,4}&:=& \mathfrak{D}_{4,2}^{\bot} \bigcap \big\{y;|y-x| \leq 2 |z-x|,\ |y-z| \leq |z-x|/2,\ |y-z|\leq t\big\}.
\end{eqnarray*}
If $y \in \mathfrak{D}_{4,2}^{\bot,1}$, then $|y-z| \geq |y-x| - |z-x| \geq |z-x|$.
If $y \in \mathfrak{D}_{4,2}^{\bot,2}$, then $|y-z| \gtrsim |z-x|$.
If $y \in \mathfrak{D}_{4,2}^{\bot,3}$, then $|y-x| \geq |z-x| - |y-z| \gtrsim |z-x| \gtrsim |y-x|$.
If $y \in \mathfrak{D}_{4,2}^{\bot,4}$, then $|z-x| \thickapprox |y-x|$.
Note also that
\begin{equation}\label{m-m}
\frac{t^{m \lambda-m}}{(t+|x-y|)^{m\lambda+2}} \leq t^{-2} \frac{t^{m + 2\alpha}}{(t+|x-y|)^{2m+2\alpha}},
\end{equation}
consequently, applying the inequalities $(\ref{y-z>t})$ and $(\ref{m-m})$, we deduce that
\begin{align*}
\mathcal{J}_2(z)^2
&\lesssim \int_{r(B)}^\infty \frac{r(B)^2}{|z-x|^{2m+2\alpha}} \int_{|y-x| > t}t^{2\alpha-2}
\frac{t^{m \lambda-m}}{(t+|x-y|)^{m\lambda}} d\mu(y) \frac{dt}{t} \\
&\quad + r(B)^2 \int_{r(B)}^\infty \int_{|y-z| > t} \frac{t^{-2}\ t^{m+2\alpha}}{(t+|z-x|)^{2m+2\alpha}} \frac{t^{2\alpha}}{(t+|y-z|)^{2m+2\alpha}} d\mu(y) \frac{dt}{t} \\
&\quad + r(B)^2 \int_{r(B)}^\infty \int_{|y-z| \leq t} \frac{t^{2\alpha-2-m}}{|z-x|^{2m+2\alpha}} d\mu(y) \frac{dt}{t} \\
&\lesssim \frac{r(B)^{2\alpha}}{|z-x|^{2m+2\alpha}}.
\end{align*}
Thus, we finish the proof of Lemma \ref{T(f)}.
\end{proof}

In order to show Lemma $\ref{good-lambda}$, it is crucial to use the Whitney decomposition given in \cite{MMT}.
\begin{lemma}\label{Whitney decomposition}
If $\Omega \subset \Rn$ is open, $\Omega \neq \Rn$, then $\Omega$ can be decomposed as
$
\Omega = \bigcup_{i \in I} Q_i
$
where $\{Q_i\}_{i \in I}$ are closed dyadic cubes with disjoint interiors such that for some
constants $\rho > 20$ and $\rho_0 \geq 1$ the following holds:
\begin{enumerate}
\item [(1)] $10 Q_i \subset \Omega$ for each $i \in I$;
\item [(2)] $\rho Q_i \cap \Omega^c \neq \emptyset$ for each $i \in I$;
\item [(3)] For each cube $Q_i$, there are at most $\rho_0$ cubes $Q_j$ such that $10 Q_i \cap 10 Q_j \neq \emptyset$.
Further, for such cubes $Q_i$, $Q_j$, we have $\ell(Q_i) \simeq \ell(Q_j)$.
\end{enumerate}
Moreover, if $\mu$ is a positive Radon measure on $\Rn$ and $\mu(\Omega)<\infty$, there is a family of cubes $\{\widetilde{Q}_j\}_{j \in S}$, with $S \subset I$,
so that $Q_j \subset \widetilde{Q}_j \subset 1.1 Q_j$ , satisfying the following:
\begin{enumerate}
\item [(a)] Each cube $\widetilde{Q}_j$, $j \in S$, is $(9, 2 \rho_0)$-doubling and has $\mathfrak{C}$-small boundary.
\item [(b)] The collection $\{\widetilde{Q}_j \}_{j \in S}$ is pairwise disjoint.
\item [(c)] it holds that
\begin{equation}\label{j-S}
\mu \Big( \bigcup_{j \in S} \widetilde{Q}_j \Big) \geq \frac{1}{8 \rho_0} \mu(\Omega).
\end{equation}
\end{enumerate}
\end{lemma}

\noindent\textbf{Proof of Lemma $\ref{good-lambda}$.}
Set
$$
\Omega_{\xi}:=\big\{x\in \Rn; g_{\lambda,\mu,t_0}^{*} (f)(x) > \xi \big\}, \quad \ \text{for any} \ \xi >0.
$$
To apply the above Whitney decomposition, we need the following claim.
\begin{claim}\label{claim}
$\Omega_{\xi} \neq \Rn$, $\mu(\Omega_{\xi})<\infty$ and $\Omega_{\xi}$ is an open set.
\end{claim}
For the sake of descriptive integrality, we postpone the proof of the claim at the end of this section.
Making use of Lemma $\ref{Whitney decomposition}$, one can get a family of dyadic cubes $\{Q_i\}_{i \in I}$ with disjoint interior such that $\Omega_{\xi} = \bigcup_{i \in I}Q_i$ and
$\rho Q_i \cap \Omega_{\xi}^c \neq \emptyset$. The collection $\{\widetilde{Q}_j \}_{j \in S}$ satisfies all properties of lemma \ref{Whitney decomposition}. From the assumption in Theorem $\ref{Theorem-L^p}$
and the fact that the cubes $\{\widetilde{Q}_j\}_{j \in S}$ have $\mathfrak{C}$-small boundary and are $(9, 2 \rho_0)$-doubling, it follows that there exists subset $G_j \subset
\widetilde{Q}_j$ with $\mu(G_j) \geq \theta \mu(\widetilde{Q}_j)$ such that $g_{\lambda}^* : \mathfrak{M}(\Rn) \rightarrow L^{1,\infty}(\mu \lfloor G_j)$, with norm bounded uniformly
on $j \in S$.
By the inequality $(\ref{j-S})$, we have
\begin{align*}
\mathfrak{F}
&:=\mu \big(\big\{x;g_{\lambda,\mu,t_0}^*(f)(x) > (1+\epsilon)\xi, M_{\mu}f(x) \leq \delta \xi  \big\}\big) \\
&\leq \mu \Big(\Omega_{\xi} \setminus \bigcup_{j \in S} \widetilde{Q}_j \Big) +  \sum_{j \in S} \mu(\widetilde{Q}_j \setminus G_j)
+ \sum_{j \in S} \mu \big(\big\{x \in G_j;g_{\lambda,\mu,t_0}^*(f)(x)> (1+\epsilon)\xi, M_{\mu}f(x) \leq \delta \xi  \big\}\big) \\
&\leq \mu (\Omega_{\xi}) - \theta \mu \Big( \bigcup_{j \in S} \widetilde{Q}_j \Big) + \sum_{j \in S} \mu \big(\big\{x \in G_j;g_{\lambda,\mu,t_0}^*(f)(x) > (1+\epsilon)\xi,
M_{\mu}f(x) \leq \delta \xi  \big\}\big) \\
&\leq \Big(1-\frac{\theta}{8 \rho_0}\Big) \mu(\Omega_\xi)  +\sum_{j \in S} \mu \big(\big\{x \in G_j;g_{\lambda,\mu,t_0}^*(f)(x) > (1+\epsilon)\xi, M_{\mu}f(x) \leq \delta \xi
\big\}\big) .
\end{align*}
In order to get the final estimate for $\mathfrak{F}$, we shall show that
\begin{equation}\label{subset}
\big\{x \in \widetilde{Q}_j;g_{\lambda,\mu,t_0}^*(f)(x) > (1+\epsilon)\xi, M_{\mu}f(x) \leq \delta \xi  \big\}
\subset \big\{x \in \widetilde{Q}_j; g_{\lambda,\mu,t_0}^*(f \mathbf{1}_{2 Q_i})(x) > \epsilon \xi/2 \big\}.
\end{equation}
Once $(\ref{subset})$ is obtained, we may deduce that
\begin{align*}
\mathfrak{F}_j &:=\mu \big(\big\{x \in G_j;g_{\lambda,\mu,t_0}^*(f)(x) > (1+\epsilon)\xi, M_{\mu}f(x) \leq \delta \xi  \big\}\big) \\
&\leq \mu \big(\big\{x \in G_j;  g_{\lambda,\mu,t_0}^*(f \mathbf{1}_{2 Q_i})(x) > \epsilon \xi/2 \big\}\big)
\leq \frac{c}{\epsilon \xi} \int_{2 \widetilde{Q}_j} |f| d\mu.
\end{align*}
If $\widetilde{Q}_j$ contains some point $x_0$ such that $M_{\mu}f(x_0) \leq \delta \xi$, then
\begin{align*}
\mathfrak{F}_i
&\leq \frac{c}{\epsilon \xi} \int_{Q(x_0,4\ell(\widetilde{Q}_j))} |f| d\mu
\leq \frac{c}{\epsilon \xi} \mu\big(Q(x_0,4\ell(\widetilde{Q}_j))\big) M_{\mu} f(x_0)\\&
\leq c \mu(10 Q_j) \delta \epsilon^{-1}
\leq 2c \rho_0 \delta \epsilon^{-1} \mu(Q_j).
\end{align*}
Hence, choosing $\delta=\delta(\epsilon)$ small enough, it yields that
$$
\mathfrak{F}
\leq \Big(1-\frac{\theta}{8 \rho_0}\Big) \mu(\Omega_\xi) + c' \sum_{j \in S} \delta \epsilon^{-1} \mu(Q_j)
\leq \Big(1-\frac{\theta}{16 \rho_0}\Big) \mu(\Omega_\xi),
$$

We are left to prove $(\ref{subset})$. Set $x \in \widetilde{Q}_j$ satisfying $g_{\lambda,\mu,t_0}^*(f)(x) > (1+\epsilon)\xi$ and
$M_{\mu}f(x) \leq \delta \xi$. It is enough to show that
\begin{equation}\label{2-Qi}
g_{\lambda,\mu,t_0}^*(f \mathbf{1}_{2 \widetilde{Q}_j})(x) > \epsilon \xi/2.
\end{equation}
Since there holds that
\begin{align*}
g_{\lambda,\mu,t_0}^*(f \mathbf{1}_{2 \widetilde{Q}_j} )(x)
\geq g_{\lambda,\mu,t_0}^*(f)(x) - g_{\lambda,\mu,t_0}^*(f \mathbf{1}_{\Rn \setminus 2 \widetilde{Q}_j})(x)
\geq (1+\epsilon) \xi - g_{\lambda,\mu,t_0}^*(f \mathbf{1}_{\Rn \setminus 2 \widetilde{Q}_j})(x),
\end{align*}
we are reduced to demonstrating
\begin{equation}\label{Rn-2Qj}
g_{\lambda,\mu,t_0}^*(f \mathbf{1}_{\Rn \setminus 2 \widetilde{Q}_j})(x) \leq (1+\epsilon/2) \xi.
\end{equation}
Take $x' \in \rho \widetilde{Q}_j \setminus \Omega_{\xi}$. We may assume that $t_0 < 2 \rho \ell(\widetilde{Q}_j)$. Then $g_{\lambda,\mu,t_0}^*(f)(x') \leq \xi$ and
\begin{equation}\label{N-N-N}
g_{\lambda,\mu,t_0}^*(f \mathbf{1}_{\Rn \setminus 2 \widetilde{Q}_j})(x)
\leq \mathfrak{N}_1 + \mathfrak{N}_2 + \mathfrak{N}_3 + \mathfrak{N}_4,
\end{equation}
where we have used the notion $\mathfrak{N}_i$ as follows:
\begin{eqnarray*}
\mathfrak{N}_1
&:=& \bigg(\int_{0}^{2 \rho \ell(\widetilde{Q}_j)} \int_{\Rn} \Big(\frac{t}{t + |x - y|}\Big)^{m \lambda} |\theta_t^\mu (f \mathbf{1}_{\Rn \setminus 2 \widetilde{Q}_j})(y)|^2
\frac{d\mu(y) dt}{t^{m+1}}\bigg)^{1/2}, \\
\mathfrak{N}_2
&:=& \bigg(\int_{2 \rho \ell(\widetilde{Q}_j)}^{\infty} \int_{\Rn} \Big(\frac{t}{t + |x' - y|}\Big)^{m \lambda} |\theta_t^\mu f(y)|^2
\frac{d\mu(y) dt}{t^{m+1}}\bigg)^{1/2}, \\
\mathfrak{N}_3
&:=& \bigg(\int_{2 \rho \ell(\widetilde{Q}_j)}^{\infty} \int_{\Rn} \Big(\frac{t}{t + |x' - y|}\Big)^{m \lambda} |\theta_t^\mu (f \mathbf{1}_{2 \widetilde{Q}_j})(y)|^2
\frac{d\mu(y) dt}{t^{m+1}}\bigg)^{1/2}, \\
\mathfrak{N}_4
&:=& \bigg| \bigg(\int_{2 \rho \ell(\widetilde{Q}_j)}^{\infty} \int_{\Rn} \Big(\frac{t}{t + |x' - y|}\Big)^{m \lambda} |\theta_t^\mu (f \mathbf{1}_{\Rn \setminus 2
\widetilde{Q}_j})(y)|^2
\frac{d\mu(y) dt}{t^{m+1}}\bigg)^{1/2} \\
&{}& -  \bigg(\int_{2 \rho \ell(\widetilde{Q}_j)}^{\infty} \int_{\Rn} \Big(\frac{t}{t + |x - y|}\Big)^{m \lambda} |\theta_t^\mu (f \mathbf{1}_{\Rn \setminus 2 \widetilde{Q}_j})(y)|^2
\frac{d\mu(y) dt}{t^{m+1}}\bigg)^{1/2} \bigg|.
\end{eqnarray*}
Since $t_0 < 2 \rho \ell(\widetilde{Q}_j)$, then $\mathfrak{N}_2 \leq g_{\lambda,\mu,t_0}^*(f)(x') \leq \xi.$
In addition, by Lemma $\ref{T(f)}$, it is easy to see that
$$
\mathfrak{N}_4 \leq \mathcal{T}(f)(x,x') \lesssim M_{\mu}f(x) \leq \delta \xi.$$
As for $\mathfrak{N}_1$, we have
\begin{align*}
\mathfrak{N}_1
= \bigg(\int_{0}^{2 \rho \ell(\widetilde{Q}_j)} \mathcal{U}_t(f \mathbf{1}_{\Rn \setminus 2 \widetilde{Q}_j})(x)^2 \frac{dt}{t}\bigg)^{1/2}.
\end{align*}
Furthermore, Lemma $\ref{U(f)}$ gives that
\begin{align*}
\mathcal{U}_t(f \mathbf{1}_{\Rn \setminus 2 \widetilde{Q}_j})(x)
&\lesssim \mathcal{V}_t(f \mathbf{1}_{\Rn \setminus 2 \widetilde{Q}_j})(x) \\
&\leq t^{\alpha} \sum_{j=0}^{\infty} \int_{2^j \ell(\widetilde{Q}_j) < |z-x| \leq 2^{j+1} \ell(\widetilde{Q}_j)} \frac{|f(z)|}{|z-x|^{m+\alpha}}  d\mu(z) \\
&\leq t^{\alpha} \sum_{j=0}^{\infty} \big( 2^j \ell(\widetilde{Q}_j) \big)^{-m-\alpha}\int_{Q(x, 2 \cdot 2^{j+1})} \frac{|f(z)|}{|z-x|^{m+\alpha}}  d\mu(z) \\
&\lesssim t^{\alpha} \ell(\widetilde{Q}_j)^{-\alpha} M_{\mu}f(x).
\end{align*}
Hence, one obtains
\begin{align*}
\mathfrak{N}_1
\lesssim \ell(\widetilde{Q}_j)^{-\alpha} M_{\mu}f(x) \bigg(\int_{0}^{2 \rho \ell(\widetilde{Q}_j)} t^{2 \alpha} \frac{dt}{t}\bigg)^{1/2}
\lesssim M_{\mu}f(x)
\leq \delta \xi.
\end{align*}
It only remains to dominate $\mathfrak{N}_3$. By Lemma $\ref{U(f)}$ again, we get
\begin{align*}
\mathfrak{N}_3
&= \Big(\int_{2 \rho \ell(\widetilde{Q}_j)}^{\infty} \mathcal{U}_t(f \mathbf{1}_{2 \widetilde{Q}_j})(x')^2 \frac{dt}{t}\Big)^{1/2} \lesssim \int_{2 \widetilde{Q}_j}|f(z)| d\mu(z) \Big(\int_{2 \rho \ell(\widetilde{Q}_j)}^{\infty} \frac {dt}{t^{2m+1}}\Big)^{1/2} \lesssim M_{\mu}f(x)
\leq \delta \xi.
\end{align*}
Collecting the above estimates, for small enough $\delta=\delta(\epsilon)$, we deduce that
$$
g_{\lambda,\mu,t_0}^*(f \mathbf{1}_{\Rn \setminus 2 \widetilde{Q}_j})(x)
\leq (1 + c \delta) \xi
\leq (1+\epsilon/2) \xi,
$$
Thus, this completes the proof of $(\ref{Rn-2Qj})$ and the proof of Lemma $\ref{good-lambda}$ is also finished.

\qed

\noindent\textbf{Proof of Claim $\ref{claim}$.}
We begin with showing $\Omega_{\xi} \neq \Rn$ and $\mu(\Omega_{\xi})<\infty$. Let $r_0 > 0$ such that $\supp f \subset B(0,r_0)$.
By Lemma $\ref{U(f)}$, it follows that for $t \geq t_0$
\begin{align*}
\mathcal{U}_t(f)(x)
&\lesssim ||f||_{L^{\infty}(\mu)} \int_{B(0,r_0)} \frac{t^{\alpha}}{(t+|x-z|)^{m+\alpha}} d\mu(z) \\
&\leq ||f||_{L^{\infty}(\mu)} \int_{B(0,r_0)} \frac{1}{(t+|x-z|)^{m}} d\mu(z) \\
&\lesssim ||f||_{L^{\infty}(\mu)} \frac{r_0^m}{(t+\dist(x,B(0,r_0)))^m} \\
&\leq ||f||_{L^{\infty}(\mu)} \frac{r_0^m}{(t_0+\dist(x,B(0,r_0)))^{m-\varepsilon}} \frac{1}{t^{\varepsilon}},
\end{align*}
where $\varepsilon \in (0,m(1-1/p))$. Then it yields that
\begin{equation}\label{C(f)}
g_{\lambda,\mu,t_0}^*(f)(x)
\leq C(f,t_0) \frac{r_0^m}{(t_0+\dist(x,B(0,r_0)))^{m-\varepsilon}}.
\end{equation}
Moreover, the above inequality implies that
$$
\big\| g_{\lambda,\mu,t_0}^*(f) \big\|_{L^p(\mu)}
\lesssim C(f,t_0) r_0^m \bigg(\int_{\Rn} \frac{d\mu(x)}{(t_0+\dist(x,B(0,r_0)))^{p(m-\varepsilon)}} \bigg)^{1/p}
< \infty.
$$
In addition, the inequality $(\ref{C(f)})$ also indicates that $\lim\limits_{|x| \rightarrow \infty}g_{\lambda,\mu,t_0}^*(f)(x) = 0$.
Thus, there exists a constant $R_0>0$ such that $\Omega_{\xi} \subset B(0,R_0)$, which implies that $\Omega_{\xi} \neq \Rn$ and $\mu(\Omega_{\xi}) < \infty$.

Next, we will show that $\Omega_{\xi}$ is an open set. This follows essentially from the fact that $x \mapsto g_{\lambda,\mu,t_0}^*(f)(x)$ is a continuous map.
The triangle inequality gives that
$$
\big| g_{\lambda,\mu,t_0}^*(f)(x) - g_{\lambda,\mu,t_0}^*(f)(x_0) \big|
\leq  \bigg(\int_{t_0}^{\infty} \big| \mathcal{U}_t(f)(x) - \mathcal{U}_t(f)(x_0) \big|^2 \frac{dt}{t}\bigg)^{1/2}.
$$
Using $\ref{U-U-V}$, for any $t \geq t_0$, we get
\begin{align*}
\big| \mathcal{U}_t(f)(x) - \mathcal{U}_t(f)(x_0) \big|
&\lesssim \frac{|x-x_0|}{t^{1-\alpha}} \int_{\Rn} \frac{|f(z)|}{(t+|x-z|)^{m+\alpha}} d\mu(z) \\
& \lesssim \frac{|x-x_0|}{t^{1-\alpha_0}} ||f||_{L^p(\mu)} \bigg(\int_{\Rn} \frac{d\mu(z)}{(t_0+|x-z|)^{(m+\alpha_0 )p'}} \bigg)^{1/{p'}} \\
& \leq C(t_0) \frac{|x-x_0|}{t^{1-\alpha_0}} ||f||_{L^p(\mu)},
\end{align*}
where the auxiliary number $\alpha_0 \in (0,1)$. Therefore, we deduce that
$$
\big| g_{\lambda,\mu,t_0}^*(f)(x) - g_{\lambda,\mu,t_0}^*(f)(x_0) \big|
\leq C(t_0) |x-x_0| ||f||_{L^p(\mu)},
$$
which implies the continuity of $x \mapsto g_{\lambda,\mu,t_0}^*(f)(x)$.
The proof of Claim $\ref{claim}$ is thus finished.

\qed
\section{Big piece $Tb$ theorem}\label{Sec-big}
In this section, our task is to demonstrate Theorem $\ref{Theorem-big}$. We need some fundamental tools including random dyadic grid and good cube used in
\cite{MMV}, which essentially traces back to \cite{NTV2003}.
\begin{definition}
Given a cube $Q \subset \Rn$, we consider the following random dyadic grid. For convenience we may assume that $c_Q = 0$.
Let $N \in \Z$ be defined by the requirement $2^{N-3} \leq \ell(Q) < 2^{N-2}$. Consider the random square
$Q^*_w = w + [-2^N, 2^N)^n$, where $w \in [-2^{N-1}, 2^{N-1})^n =: \Omega_N = \Omega$. The set $\Omega$ is equipped with the normalised Lebesgue measure $\mathbb{P}_N = \mathbb{P}$.
We define the grid $\D_w := \D(Q^*_w)$ (the local dyadic grid generated by the cube $Q^*_w$). Notice that $Q \subset \eta Q^*_w$ for some $\eta < 1$, and $\ell(Q) \simeq \ell(Q^*_w)$.
\end{definition}

\begin{definition}
A cube $I \in \mathcal{D}_0$ is said to be $\D_w$-good if there exists a $J \in \D_w$ with $\ell(J) \geq 2^r \ell(I)$ such that $\dist(I,\partial J) > \ell(I)^{\gamma}
\ell(J)^{1-\gamma}$. Otherwise, $I$ is called $\D_w$-bad.
Here $r \in \Z_+$ is a fixed large enough parameter, and $\gamma =\frac{\alpha}{2(m+\alpha)}$.
\end{definition}
\subsection{Proof of the first result}\label{subsec-first}
First, we will show that the result $(i)$ in Theorem $\ref{Theorem-big}$ is true.
Set $S_0 :=\big\{x \in Q; g_{\lambda,\sigma,Q}^*(b)(x) > \xi_0 \big\}\ \text{for fixed}\ \xi_0 \in (0,\infty)$.
Then, if we take $\xi_0$ large enough such that $\xi_0^s > 2C_1/(1-\delta_0)$, there holds that
\begin{align*}
\sigma(H \cup T_w \cup S_0)
&\leq \sigma(H \cup T_w) + \sigma(S_0 \setminus H)
\leq \delta_0 \sigma(Q) + \sigma \big(\big\{x \in Q \setminus H; g_{\lambda,\sigma,Q}^*(b)(x) > \xi_0 \big\}\big)\\
&\leq \delta_0 \sigma(Q) + C_1 \xi_0^{-s} \sigma(Q)
\leq (1+\delta_0)/2 \ \sigma(Q)
:=\delta_1 \sigma(Q),
\end{align*}
Let $\tau=(1-\delta_1)/2$, we now introduce the notions $G_Q$ and $P(x)$:
$$
G_Q := \{x \in Q;\ P(x) > \tau \}\ \text{and}\
P(x)=\mathbb{P}\big(\{w \in \Omega; x \in Q \setminus (H \cup T_w \cup S_0) \}\big).
$$
Then, we have
\begin{align*}
(1-\tau)\sigma(Q \setminus G_Q)
&\leq \int_{Q \setminus G_Q} (1-P(x)) d\sigma(x)
\leq \int_{Q} (1-P(x)) d\sigma(x) \\
&=\sigma(Q)-\int_{\Omega} \sigma(H \cup T_w \cup S_0) d\mathbb{P}(w) \\
&\leq \sigma(Q)-(1-\tau)\sigma(Q \setminus G_Q).
\end{align*}
Therefore, it holds that
$$
\sigma(Q) \leq \frac{2-\tau}{1-\tau}\sigma(G_Q).
$$
\qed
\subsection{The Probabilistic Reduction}
From now on, we will give the proof of $(ii)$ in Theorem $\ref{Theorem-big}$. Some reductions will be made, and we need to discard the bad cubes and consider the contributions of good cubes.
\subsubsection{\textbf{Discarding bad cubes.}}
We may assume that $\big\| \mathbf{1}_{G_Q} g_{\lambda,\sigma,Q}^* \big\|_{L^2(\sigma)} < \infty$. Indeed, adopting the similar methods used in Proposition $3.1$ \cite{CX}, we can obtain the priori assumption. Then
\begin{align*}
\big\| \mathbf{1}_{G_Q} g_{\lambda,\sigma,Q}^* \big\|_{L^2(\sigma)}^2
&=\int_{G_Q} \int_{0}^{\ell(Q)} \int_{\Rn} \Big(\frac{t}{t+|x-y|}\Big)^{m\lambda} |\theta_t^\sigma f(y)|^2 \frac{d\mu(y) dt}{t^{m+1}} d\sigma(x) \\
&\leq \tau^{-1} \mathbb{E}_w \int_{G_Q \setminus [H \cap T_w \cap S_0]} \int_{0}^{\ell(Q)} \int_{\Rn} \Big(\frac{t}{t+|x-y|}\Big)^{m\lambda} |\theta_t^\sigma f(y)|^2 \frac{d\mu(y)
dt}{t^{m+1}} d\sigma(x) \\
&= \tau^{-1} \mathbb{E}_w \sum_{R \in \mathcal{D}_0} \int_{(R \cap G_Q) \setminus (H \cap T_w \cap S_0)} \int_{\ell(R)/2}^{\min\{\ell(R),\ell(Q)\}} \int_{\Rn}
\Big(\frac{t}{t+|x-y|}\Big)^{m\lambda} \\&\quad \times|\theta_t^\sigma f(y)|^2
\frac{d\mu(y) dt}{t^{m+1}} d\sigma(x).
\end{align*}
Next, we try to bound the summation over the bad cubes.
\begin{align*}
&\mathbb{E}_w \sum_{\substack{R \in \mathcal{D}_0 \\ R : \D_w{-bad}}} \int_{(R \cap G_Q) \setminus (H \cap T_w \cap S_0)} \int_{\ell(R)/2}^{\min\{\ell(R),\ell(Q)\}} \int_{\Rn}
\Big(\frac{t}{t+|x-y|}\Big)^{m\lambda} |\theta_t^\sigma f(y)|^2
\frac{d\mu(y) dt}{t^{m+1}} d\sigma(x)\end{align*}\begin{align*}
&\leq \sum_{R \in \mathcal{D}_0} \mathbb{P}\big(\{w \in \Omega; R \ \text{is}\ \D_w{-bad} \}\big)\int_{R \cap G_Q} \int_{\ell(R)/2}^{\min\{\ell(R),\ell(Q)\}} \int_{\Rn}
\Big(\frac{t}{t+|x-y|}\Big)^{m\lambda} \\&\quad\times|\theta_t^\sigma f(y)|^2
\frac{d\mu(y) dt}{t^{m+1}} d\sigma(x)\\
&\leq \frac{\tau}{2} \int_{G_Q} \int_{0}^{\ell(Q)} \int_{\Rn} \Big(\frac{t}{t+|x-y|}\Big)^{m\lambda} |\theta_t^\sigma f(y)|^2 \frac{d\mu(y) dt}{t^{m+1}} d\sigma(x)
=\frac{\tau}{2} \big\| \mathbf{1}_{G_Q} g_{\lambda,\sigma,Q}^* \big\|_{L^2(\sigma)}^2,
\end{align*}
where we used the fact (see \cite{NTV2003})
$$
\mathbb{P}\big(\{w\in \Omega; R \ \text{is}\ \D_w-bad \}\big) \leq \tau/2.
$$

Since $\big\| \mathbf{1}_{G_Q} g_{\lambda,\sigma,Q}^* \big\|_{L^2(\sigma)} < \infty$, we deduce that
\begin{align*}
&\big\| \mathbf{1}_{G_Q} g_{\lambda,\sigma,Q}^* \big\|_{L^2(\sigma)}^2 \\
&\lesssim \mathbb{E}_w \sum_{\substack{R \in \mathcal{D}_0 \\ R : \D_w{-good}}} \int_{(R \cap G_Q) \setminus (H \cap T_w \cap S_0)} \int_{\ell(R)/2}^{\min\{\ell(R),\ell(Q)\}}
\int_{\Rn} \Big(\frac{t}{t+|x-y|}\Big)^{m\lambda}\\&\quad\times |\theta_t^\sigma f(y)|^2
\frac{d\mu(y) dt}{t^{m+1}} d\sigma(x)\\
&=\mathbb{E}_w \sum_{\substack{R \in \mathcal{D}_0 \\ R : \D_w{-good} \\ R \not\subset H \cup T_w}} \int_{R} \int_{\ell(R)/2}^{\min\{\ell(R),\ell(Q)\}} \int_{\Rn}
\Big(\frac{t}{t+|x-y|}\Big)^{m\lambda} |\widetilde{\theta}_t^\sigma f(y)|^2
\frac{d\mu(y) dt}{t^{m+1}} d\sigma(x),
\end{align*}
where $\widetilde{\theta}_t^\sigma f(y)=\int_{\Rn} \widetilde{s}_{t}(y,z)f(z)d\sigma(z)$
and $\widetilde{s}_t(y,z)=s_t(y,z) \mathbf{1}_{\Rn \setminus S_0}(x)$. It is easy to check that $\widetilde{s}_t$ satisfies the Size condition $(1)$ and H\"{o}lder conditions $(2)$.

From now on, $w$ is fixed, simply denote $\mathcal{D}=\D_w$ and $T=T_w$. It is enough to prove that
\begin{equation}\label{Reduction-1}
\sum_{\substack{R \in \mathcal{D}_0 \\ R : \mathcal{D}{-good} \\ R \not\subset H \cup T}} \int_{R} \int_{\ell(R)/2}^{\min\{\ell(R),\ell(Q)\}} \int_{\Rn}
\Big(\frac{t}{t+|x-y|}\Big)^{m\lambda} |\widetilde{\theta}_t^\sigma f(y)|^2
\frac{d\mu(y) dt}{t^{m+1}} d\sigma(x)
\lesssim ||f||_{L^2(\sigma)}^2.
\end{equation}
\qed
\subsubsection{\textbf{Martingale difference operators.}}
To get further reduction, we introduce $b$-adapted martingales only in the transit cubes $P \in \mathcal{D}^{tr} = \mathcal{D}_w^{tr}$, which is defined by
$$
\mathcal{D}^{tr}_w := \big\{P \in \mathcal{D}_w;\ \sigma(P) \neq 0 \text{ and } P \not\subset H \cup T_w \big\}.
$$
Without loss of generality, we may assume that $\supp b \subset Q$ and $\supp f \subset Q$.
We introduce the notations :
$$
\mathbb{E}_{P_0}f(x) = \frac{\langle f \rangle_{P_0}}{\langle b \rangle_{P_0}} b(x), \
\langle f \rangle_{E} = \frac{1}{\sigma(E)}\int_E f d\sigma, \ \text{if} \ \sigma(E)\neq 0,
$$
where $P_0=Q^*_w \in \D^{tr}$. Then for any cube $P \in \D^{tr}$, the martingale difference operators $\Delta_P f$ are defined by:
$$
\Delta_P f = \sum_{P' \in ch(P)} A_{P'}(f) \mathbf{1}_{P'},
$$
where
\begin{equation*}
A_{P'}(f) =
\begin{cases}
\Big(\frac{\langle f \rangle_{P'}}{\langle b \rangle_{P'}}- \frac{\langle f \rangle_{P}}{\langle b \rangle_{P}}\Big)b \ \ & \text{if} \ P' \in \mathcal{D}^{tr}, \\
f - \frac{\langle f \rangle_{P}}{\langle b \rangle_{P}}b \ & \text{if} \ P' \not\in \mathcal{D}^{tr}.
\end{cases}
\end{equation*}
For convenience, we set $\Delta_{P_0}f$ to be $\Delta_{P_0} f + \mathbb{E}_{P_0}f$ at the largest level $P_0$.
Then there holds that (see Lemma 5.11 \cite{Tolsa})
$$
f=\sum_{P \in \mathcal{D}^{tr}} \Delta_P f, \ \
\sum_{P \in \mathcal{D}^{tr}} \big\| \Delta_P f \big\|_{L^2(\sigma)}^2 \lesssim \big\| f \big\|_{L^2(\sigma)}^2 .
$$
Applying the above expansion of $f$, the inequality $(\ref{Reduction-1})$ turns into the following inequality
\begin{equation}\label{Reduction-2}
\sum_{\substack{R \in \mathcal{D}^{tr}_0 \\ R : \mathcal{D}{-good}}} \int_{R} \int_{\ell(R)/2}^{\min\{\ell(R),\ell(Q)\}} \int_{\Rn} \big(\frac{t}{t+|x-y|}\big)^{m\lambda} \big|
\sum_{P \in \mathcal{D}^{tr}}\widetilde{\theta}_t^\sigma (\Delta_P f)(y)\big|^2 \frac{d\mu(y) dt}{t^{m+1}} d\sigma(x)
\lesssim ||f||_{L^2(\sigma)}^2.
\end{equation}

\subsection{Main estimates.}
In this section, we shall give the demonstration of the inequality $(\ref{Reduction-2})$. For a fixed cube
$R \in \D_0^{tr} \cap \D-{good}$, we begin by splitting the transit cubes $P \in \mathcal{D}^{tr}$ into four different cases:
\begin{enumerate}
\item [$\bullet$] Less : $\ell(P) < \ell(R)$;
\vspace{0.2cm}
\item [$\bullet$] Separated : $\ell(P) \geq \ell(R)$, $d(P,R) > \ell(R)^{\gamma} \ell(P)^{1-\gamma}$;
\vspace{0.2cm}
\item [$\bullet$] Nearby : $\ell(R) \leq \ell(P) \leq 2^r \ell(R)$, $d(P,R) \leq \ell(R)^{\gamma}\ell(P)^{1-\gamma}$;
\vspace{0.2cm}
\item [$\bullet$] Inside : $\ell(P) > 2^r \ell(R)$, $d(P,R) \leq \ell(R)^{\gamma} \ell(P)^{1-\gamma}$.
\end{enumerate}
Thus, the left hand side of $(\ref{Reduction-2})$ is dominated by the correspondingly four pieces, which are denoted by
$\Sigma_{less}$,$\Sigma_{sep}$,$\Sigma_{near}$ and $\Sigma_{in}$ . We will treat these four terms respectively.

\subsubsection{\textbf{Less part.}}\label{sec-less}
Before starting the proof, we first present two key lemmas.
\begin{lemma}[\cite{Tolsa}]\label{A-QR}
Denote $$A_{QR}=\frac{\ell(Q)^{\alpha/2} \ell(R)^{\alpha/2}}{D(Q,R)^{m+\alpha}} \sigma(Q)^{1/2} \sigma(R)^{1/2},$$
where $D(Q,R)=\ell(Q) + \ell(R) + d(Q,R)$, $Q \in \mathcal{D}^{tr}$, $R \in \mathcal{D}^{tr}_0$ and $\alpha > 0$. Then for every
$x_Q \geq 0$, there holds that
$$
\sum_{R \in \mathcal{D}^{tr}_0} \Big( \sum_{Q \in \mathcal{D}^{tr}} A_{QR} x_{Q} \Big)^2
\lesssim \sum_{Q \in \mathcal{D}^{tr}_0} x_{Q}^2.
$$
\end{lemma}
\begin{lemma}\label{estimate-1}
Let $0 < \alpha \leq m(\lambda - 2)/2$. Suppose that $\ell(P) < \ell(R)$ and $(x,t) \in W_R$. Then it holds that
\begin{equation*}
\bigg( \int_{\Rn} |\widetilde{\theta}_t^{\sigma}(\Delta_P f)(y)|^2 \Big(\frac{t}{t+|x-y|}\Big)^{m\lambda}\frac{d\mu(y)}{t^m}\bigg)^{1/2}
\lesssim A_{PR} \sigma(R)^{-1/2} \big\| \Delta_P f \big\|_{L^2(\sigma)}.
\end{equation*}
\end{lemma}

\begin{proof}
Let $z_P$ be the center of $P$. Since $\ell(P) < \ell(R) \leq \ell(P_0)$, we get the vanishing property
$\int_P \Delta_P f d\sigma =0$. Thus, we have
$$|\widetilde{\theta}_t(\Delta_P f)(y)|
= \bigg|\int_{P}[s_t(y,z)-s_t(y,z_P)] \Delta_P f(z) d\sigma(z)\bigg|
\lesssim \int_{P} \frac{\ell(P)^\alpha}{(t+|y-z|)^{m+\alpha}}|\Delta_P f(z)| d\sigma(z). $$
For $x \in R$ and $z \in P$, $|x-z|\geq d(P,R)$. We will consider two subcases.

First, we analyze the contribution of the subregion in which $y:|y-x|\leq \frac12 d(P,R)$.
In this case, $|y-z| \geq |x-z|-|x-y| \gtrsim d(P,R)$ and $\ell(R) + d(P,R) \simeq  D(P,R)$.
Thus, it follows that
$$|\widetilde{\theta}_t^\sigma(\Delta_P f)(y)|
\lesssim \frac{\ell(P)^\alpha}{D(P,R)^{m+\alpha}} \big\| \Delta_P f \big\|_{L^1(\sigma)}¡£$$
Therefore,
\begin{align*}
\bigg( \int_{y:|y-x|\leq \frac12 d(P,R)} |\widetilde{\theta}_t^\sigma(\Delta_P f)(y)|^2 \Big(\frac{t}{t+|x-y|}\Big)^{m\lambda}\frac{d\mu(y)}{t^m}\bigg)^{1/2}
\lesssim A_{PR}\ \sigma(R)^{-1/2} \big\| \Delta_P f \big\|_{L^2(\sigma)}.
\end{align*}

Secondly, we treat the contribution made by those $y:|y-x| > \frac12 d(P,R)$. Since $|y-x| > \frac12 d(P,R)$, we have
$$ \frac{t}{t+|x-y|} \lesssim \frac{\ell(R)}{\ell(R)+d(P,R)}.$$
Accordingly, together with size condition and Young's inequality, we conclude that
\begin{align*}
&\bigg(\int_{y:|y-x|> \frac12 d(P,R)} |\widetilde{\theta}_t^\sigma(\Delta_P f)(y)|^2 \Big(\frac{t}{t+|x-y|}\Big)^{m\lambda}\frac{d\mu(y)}{t^m}\bigg)^{1/2} \\
&\lesssim \frac{\ell(R)^{\frac{m\lambda}{2}-m}}{(\ell(R)+d(P,R))^{\frac{m\lambda}{2}}}  t^{-m/2}
\bigg[\int_{\Rn}\bigg(\int_{\Rn}\frac{\ell(P)^{\alpha}}{(t+|y-z|)^{m+\alpha}} |\Delta_P f(z)| d\sigma(z)\bigg)^2 d\sigma(y)\bigg]^{1/2} \\
&\leq \frac{\ell(P)^\alpha \ell(R)^\alpha}{(\ell(R)+d(P,R))^{m+\alpha}} t^{-m/2}  \Big\|\frac{1}{(t+|\cdot|)^{m+\alpha}}\Big\|_{L^2(\mu)} \big\|\Delta_P f \big\|_{L^1(\sigma)} \\
&\lesssim \frac{\ell(P)^\alpha}{D(P,R)^{m+\alpha}} \big\| \Delta_P f \big\|_{L^1(\sigma)}
\leq A_{PR}\ \sigma(R)^{-1/2} \big\| \Delta_P f \big\|_{L^2(\sigma)},
\end{align*}
where we have used the condition $0 < \alpha \leq m(\lambda - 2)/2$.

The proof of Lemma $\ref{estimate-1}$ is completed.
\end{proof}

By the H\"{o}lder inequality, Lemma $(\ref{estimate-1})$ and Lemma $\ref{A-QR}$, it follows that
\begin{align*}
\Sigma_{less}
&\leq \sum_{\substack{R \in \mathcal{D}^{tr}_0 \\ R : \mathcal{D}-good \\ \ell(R) \leq 2^s}} \iint_{W_R}
\bigg[\sum_{\substack{P \in \mathcal{D}^{tr} \\ \ell(P) < \ell(R)}} \bigg(\int_{\Rn} |\widetilde{\theta}_t^\sigma(\Delta_P f)(y)|^2
\Big(\frac{t}{t+|x-y|}\Big)^{m\lambda}\frac{d\mu(y)}{t^m}\bigg)^{1/2}\bigg]^2 d\sigma(x) \frac{dt}{t} \\
&\lesssim \sum_{R \in \mathcal{D}^{tr}_0}
\bigg(\sum_{P \in \mathcal{D}^{tr}} A_{PR} \big\| \Delta_P f \big\|_{L^2(\sigma)}\bigg)^2
\lesssim \sum_{P \in \mathcal{D}^{tr}} \big\| \Delta_P f \big\|_{L^2(\sigma)}^2
\lesssim \big\| f \big\|_{L^2(\sigma)}^2 .
\end{align*}

\subsubsection{\textbf{Separated part.}}\label{sec-sep}
Similar to the proof in the preceding subsection, it is sufficient to establish the following lemma.
\begin{lemma}\label{estimate-2}
Let $0 < \alpha \leq m(\lambda - 2)/2$. Assume that $P \in \D^{tr}$ and $R \in \D_0^{tr} \cap \D-good$ are given cubes satisfying $\ell(P) \geq \ell(R)$, $d(P,R) > \ell(R)^{\gamma}
\ell(P)^{1-\gamma}$. Then for any $(x,t) \in W_R$ we have
\begin{equation*}
\bigg( \int_{\Rn} |\widetilde{\theta}_t^\sigma(\Delta_P f)(y)|^2 \Big(\frac{t}{t+|x-y|}\Big)^{m\lambda}\frac{d\mu(y)}{t^m}\bigg)^{1/2}
\lesssim A_{PR}\ \sigma(R)^{-1/2} \big\| \Delta_P f \big\|_{L^2(\sigma)}.
\end{equation*}
\end{lemma}

\begin{proof}
We begin by showing that
\begin{equation}\label{R-R-P-R}
\frac{\ell(R)^\alpha}{(\ell(R)+d(P,R))^{m+\alpha}}
\lesssim \frac{\ell(P)^{\alpha/2} \ell(R)^{\alpha/2}}{D(P,R)^{m+\alpha}}.
\end{equation}
Actually, if $\ell(P) \leq d(P,R)$, it is obvious that
\begin{align*}
\frac{\ell(R)^\alpha}{(\ell(R)+d(P,R))^{m+\alpha}}
\lesssim \frac{\ell(R)^\alpha}{D(P,R)^{m+\alpha}}
\leq \frac{\ell(P)^{\alpha/2} \ell(R)^{\alpha/2}}{D(P,R)^{m+\alpha}}.
\end{align*}
If $\ell(P) > d(P,R)$, then $D(P,R) \simeq \ell(P)$. Using $d(P,R) > \ell(R)^{\gamma} \ell(P)^{1-\gamma}$ and
$\gamma = \frac{\alpha}{2(m+\alpha)}$, we obtain
\begin{align*}
\ell(P) = \bigg(\frac{\ell(P)}{\ell(R)}\bigg)^\gamma \ell(R)^\gamma \ell(P)^{1-\gamma}
< \bigg(\frac{\ell(P)}{\ell(R)}\bigg)^\gamma d(P,R),
\end{align*}
and
\begin{align*}
\frac{\ell(R)^\alpha}{(\ell(R)+d(P,R))^{m+\alpha}}
\leq \frac{\ell(R)^\alpha}{d(P,R)^{m+\alpha}}
\leq \frac{\ell(P)^{\alpha/2} \ell(R)^{\alpha/2}}{\ell(P)^{m+\alpha}}
\simeq \frac{\ell(P)^{\alpha/2} \ell(R)^{\alpha/2}}{D(P,R)^{m+\alpha}}.
\end{align*}

Next, we continue with the proof. Using the size condition, we get
$$|\widetilde{\theta}_t^\sigma(\Delta_P f)(y)|
\lesssim \int_{\Rn}\frac{t^\alpha}{(t+|y-z|)^{m+\alpha}} |\Delta_P f(z)|d\sigma(z). $$
A similar arguments as that in Lemma $\ref{estimate-1}$ may yield that
\begin{equation*}
\bigg( \int_{\Rn} |\widetilde{\theta}_t^{\sigma}(\Delta_P f)(y)|^2 \Big(\frac{t}{t+|x-y|}\Big)^{m\lambda}\frac{d\mu(y)}{t^m}\bigg)^{1/2}
\lesssim A_{PR}\ \sigma(R)^{-1/2} \big\| \Delta_P f \big\|_{L^2(\sigma)}.
\end{equation*}
This shows that Lemma $\ref{estimate-2}$ is true.
\end{proof}

\subsubsection{\textbf{Nearby part.}}\label{sec-near}
In this case, it is trivial that $D(P,R) \simeq \ell(P) \simeq \ell(R)$. Hence, we have
$$
\frac{\ell(R)^\alpha}{(\ell(R)+d(P,R))^{m+\alpha}}
\leq \ell(R)^{-m} \simeq \frac{\ell(P)^{\alpha/2} \ell(R)^{\alpha/2}}{D(P,R)^{m+\alpha}} .
$$
This  parallels with $(\ref{R-R-P-R})$. A completely analogous calculation to that of the preceding section yields that
$\Sigma_{near} \lesssim \big\| f \big\|_{L^2(\sigma)}^2$.
\subsubsection{\textbf{Inside part.}}
Let $R^{(k)} \in \mathcal{D}$ be the unique cube for which $\ell(R^{(k)}) = 2^k \ell(R)$ and $R \subset R^{(k)}$. We call $R^{(k)}$ as the $k$ generations older dyadic
ancestor of $R$. In this case, since $R$ is good, it must actually have
$R \subset P$. That is, $P$ is the ancestor of $R$. Then we can write
\begin{align*}
\Sigma_{in}
&=\sum_{\substack{R \in \mathcal{D}_0^{tr}:R \subset P \\ R : \mathcal{D}-good \\ \ell(R) \leq 2^{-r} \ell(P_0)}}
\int_{R}\int_{\ell(R)/2}^{\min\{\ell(R),\ell(Q)\}} \int_{\Rn} \Big|\sum_{k=r+1}^{\log_2 \frac{\ell(P_0)}{\ell(R)}} \widetilde{\theta}_t^{\sigma}(\Delta_{R^{(k)}} f)(y)\Big|^2\\&\quad\times
\Big(\frac{t}{t+|x-y|}\Big)^{m \lambda}\frac{d\mu(y)}{t^m} \frac{dt}{t} d\sigma(x) \\
\end{align*}
Note that on the largest level $P_0$, the operator $\Delta_{P_0}$ is defined by $\Delta_{P_0} + \mathbb{E}_{P_0}$. Hence,
$$ B_{R^{(k-1)}}:=\langle b^{-1} \Delta_{R^{(k)}} f \rangle_{R^{(k-1)}}
=\begin{cases}
\frac{\langle f \rangle_{R^{(k-1)}}}{\langle b \rangle_{R^{(k-1)}}} - \frac{\langle f \rangle_{R^{(k)}}}{\langle b \rangle_{R^{(k)}}}, \ &\text{if}\ r+1 \leq k < \log_2
\frac{\ell(P_0)}{\ell(R)}, \\
\frac{\langle f \rangle_{R^{(k-1)}}}{\langle b \rangle_{R^{(k-1)}}},   \ &\text{if}\ k = \log_2 \frac{\ell(P_0)}{\ell(R)} .
\end{cases}
$$
And hence, we obtain the decomposition
$$ \Delta_{R^{(k)}}f
= - B_{R^{(k-1)}} \mathbf{1}_{(R^{(k-1)})^c} b
  + \sum_{\substack{S \in ch(R^{(k)}) \\ S \neq R^{(k-1)}}} \mathbf{1}_S \Delta_{R^{(k)}}f
  + B_{R^{(k-1)}} b .$$
Observing that
$$ \sum_{k=r+1}^{\log_2 \frac{\ell(P_0)}{\ell(R)}} B_{R^{(k-1)}}
= \frac{\langle f \rangle_{R^{(r)}}}{\langle b \rangle_{R^{(r)}}}.$$
Thus, $\Sigma_{in}$ is dominated
$$ \Sigma_{in} \lesssim \Sigma_{in}^{'} + \Sigma_{in}^{''} + \Sigma_{in}^{'''},$$
where
\begin{align*}
\Sigma_{in}^{'}
&= \sum_{\substack{R \in \mathcal{D}_0^{tr}:R \subset P \\ R : \mathcal{D}-good \\ \ell(R) \leq 2^{-r} \ell(P_0)}}
\int_{R}\int_{\ell(R)/2}^{\min\{\ell(R),\ell(Q)\}} \int_{\Rn} \Big|\sum_{k=r+1}^{\log_2 \frac{\ell(P_0)}{\ell(R)}} B_{R^{(k-1)}} \widetilde{\theta}_t^\sigma(\mathbf{1}_{(R^{(k-1)})^c}
b)(y)\Big|^2 \\&\quad\times\big(\frac{t}{t+|x-y|}\big)^{m \lambda}\frac{d\mu(y)dt}{t^{m+1}} d\sigma(x) ,
\end{align*}
\begin{align*}
\Sigma_{in}^{''}
&= \sum_{\substack{R \in \mathcal{D}_0^{tr}:R \subset P \\ R : \mathcal{D}-good \\ \ell(R) \leq 2^{-r} \ell(P_0)}}
\int_{R}\int_{\ell(R)/2}^{\min\{\ell(R),\ell(Q)\}} \int_{\Rn} \Big|\sum_{k=r+1}^{\log_2 \frac{\ell(P_0)}{\ell(R)}}
\sum_{\substack{S \in ch(R^{(k)}) \\ S \neq R^{(k-1)}}} \widetilde{\theta}_t^\sigma(\mathbf{1}_{S} \Delta_{R^{(k)}}f)(y)\Big|^2\\&\quad\times \big(\frac{t}{t+|x-y|}\big)^{m \lambda}\frac{d\mu(y)dt}{t^{m+1}}
d\sigma(x) ,
\end{align*}
and
\begin{align*}
\Sigma_{in}^{'''}
&= \sum_{\substack{R \in \mathcal{D}_0^{tr}:R \subset P \\ R : \mathcal{D}-good \\ \ell(R) \leq 2^{-r} \ell(P_0)}}
\frac{|\langle f \rangle_{R^{(r)}}|^2}{|\langle b \rangle_{R^{(r)}}|^2}
\int_{R}\int_{\ell(R)/2}^{\min\{\ell(R),\ell(Q)\}} \int_{\Rn} \big|\widetilde{\theta}_t^{\sigma} b(y)\big|^2 \big(\frac{t}{t+|x-y|}\big)^{m \lambda}\frac{d\mu(y)}{t^m}\frac{dt}{t}
d\sigma(x) .
\end{align*}
\noindent\textbf{$\bullet$ Estimate of $\Sigma_{in}^{'}$.}
In this case, the key point is to gain a geometric decay in $k$ by using the goodness of the cubes. First, we need the following lemma.
\begin{lemma}\label{alpha-k/2}
Let $0 < \alpha \leq m(\lambda - 2)/2$. Given a cube $R \in \mathcal{D}_{good}$ and $k \geq r+1$, it holds that
\begin{equation*}
\bigg(\iint_{W_R} \int_{\Rn} |\widetilde{\theta}_t^\sigma(\mathbf{1}_{(R^{(k-1)})^c} b)(y)|^2 \Big(\frac{t}{t+|x-y|}\Big)^{m\lambda}
\frac{d\mu(y)}{t^m} d\sigma(x) \frac{dt}{t} \bigg)^{1/2}
\lesssim 2^{- \alpha k/2}\sigma(R)^{1/2}.
\end{equation*}
\end{lemma}

\begin{proof}
We only need to show
\begin{align*}
\mathscr{K}:= \bigg(\int_{\Rn} |\widetilde{\theta}_t^\sigma(\mathbf{1}_{(R^{(k-1)})^c} b)(y)|^2 \Big(\frac{t}{t+|x-y|}\Big)^{m\lambda}
\frac{d\mu(y)}{t^m} \bigg)^{1/2} \lesssim 2^{-\alpha k/2}.
\end{align*}
We start with
\begin{align*}
\mathscr{K} & \lesssim \bigg[ \int_{\Rn} \bigg(\int_{R^{(k-1)^c}} \frac{t^\alpha}{(t + |y - z|)^{m + \alpha}} d\sigma(z) \bigg)^2 \Big(\frac{t}{t + |x-y|}\Big)^{m \lambda}
\frac{d\mu(y)}{t^m} \bigg]^{1/2} \\
&\leq \bigg[ \int_{\Rn} \bigg(\int_{E_1} \frac{t^\alpha}{(t + |y - z|)^{m + \alpha}} d\sigma(z) \bigg)^2 \Big(\frac{t}{t + |x-y|}\Big)^{m \lambda} \frac{d\mu(y)}{t^m} \bigg]^{1/2} \\
&\quad + \bigg[ \int_{\Rn} \bigg(\int_{E_2} \frac{t^\alpha}{(t + |y - z|)^{m + \alpha}} d\sigma(z) \bigg)^2 \Big(\frac{t}{t + |x-y|}\Big)^{m \lambda} \frac{d\mu(y)}{t^m} \bigg]^{1/2}:= \mathscr{K}_1 + \mathscr{K}_2,
\end{align*}
where $ E_1 = \big\{z \in (R^{(k-1)})^c; |z - x| \geq 2|x - y| \big\}$ and $E_2 = \big\{z \in (R^{(k-1)})^c; |z - x| < 2|x - y| \big\}.$
Since $k > r$, by the goodness of $R$, one obtains
$$ d(R,(R^{(k-1)})^c) > \ell(R)^{\gamma} \ell(R^{(k-1)})^{1-\gamma}= 2^{(k-1)(1-\gamma)} \ell(R) \gtrsim 2^{k/2} \ell(R).$$
It is should be noted that $R \subset B(x,d(R,\partial R^{(k-1)})/2) \not\subset H$ whenever $r$ is large enough.
Thus, by splitting the domain of integration into annuli and using the condition $(b)$ in Theorem $\ref{Theorem-big}$, it yields that
\begin{align*}
\int_{(R^{(k-1)})^c} \frac{\ell(R)^\alpha}{|z - x|^{m+\alpha}} d\sigma(z)
&\leq \int_{\Rn \setminus B(x,d(R,\partial R^{(k-1)}))} \frac{\ell(R)^\alpha}{|z - x|^{m+\alpha}} d\sigma(z) \\
&\lesssim \ell(R)^\alpha d(R, \partial R^{(k-1)})^{-\alpha} \lesssim 2^{-\alpha k/2}.
\end{align*}
For $\mathscr{K}_1$, if $z \in E_1$, then $ t + |y - z| > |x - z| - |x - y| \geq |x - z|/2$.
Thus, we obtain
$$ \mathscr{K}_1
\lesssim \int_{(R^{(k-1)})^c} \frac{\ell(R)^\alpha}{|z - x|^{m+\alpha}}d\sigma(z) \cdot
\bigg[ \int_{\Rn} \Big(\frac{t}{t + |x-y|}\Big)^{m \lambda} \frac{d\mu(y)}{t^m} \bigg]^{1/2}
\lesssim 2^{-\alpha k/2}.$$
As for $\mathscr{K}_2 $, by the Young inequality, we have the following estimate,
\begin{align*}
\mathscr{K}_2
&\lesssim \bigg[ \int_{\Rn} \bigg(\int_{E_2} \Big(\frac{t}{t+|y-z|}\Big)^{m+\alpha} \frac{t^{\frac{m \lambda}{2}-m}}{(t + |x - z|)^{\frac{m \alpha}{2}}} d\sigma(z) \bigg)^2
\frac{d\mu(y)}{t^m} \bigg]^{1/2} \\
&\lesssim \bigg[ \int_{\Rn} \bigg(\int_{E_2} \Big(\frac{t}{t+|y-z|}\Big)^{m+\alpha} \frac{t^\alpha}{(t + |x - z|)^{m+\alpha}} d\sigma(z) \bigg)^2 \frac{d\mu(y)}{t^m} \bigg]^{1/2} \\
&=\big\| \xi * \eta \big\|_{L^2(\mu)}
\leq \big\| \xi \big\|_{L^2(\mu)} \big\| \eta \big\|_{L^1(\sigma)} \\
&\lesssim \int_{(R^{(k-1)})^c} \frac{\ell(R)^\alpha}{|z - x|^{m+\alpha}}d\sigma(z)
\lesssim 2^{-\alpha k/2},
\end{align*}
where
$ \xi(z) = \Big(\frac{t}{t + |z|}\Big)^{m + \alpha}$ and $\eta(z) = \frac{t^\alpha}{(t + |z-x|)^{m + \alpha}}\mathbf{1}_{E_2}(z).$
\end{proof}

We are in the position of dominating $\Sigma_{in}^{'}$.
By $R^{(k-1)} \not\subset T$ and the accretivity condition for $b$, it follows that
\begin{align*}
|B_{R^{(k-1)}}|
&\lesssim \sigma(R^{(k-1)})^{-1} \bigg| \int_{R^{(k-1)}} B_{R^{(k-1)}} b(x) d\sigma(x) \bigg| \\
&\lesssim \sigma(R^{(k-1)})^{-1} \bigg| \int_{R^{(k-1)}} \Delta_{R^{(k)}} f(x) d\sigma(x) \bigg| \\
&\lesssim \sigma(R^{(k-1)})^{-1/2} \big\| \Delta_{R^{(k)}} f \big\|_{L^2(\sigma)}.
\end{align*}
Therefore, the Minkowski inequality, together with Lemma $\ref{alpha-k/2}$ yields that
\begin{align*}
\Sigma_{in}^{'}
&\lesssim \sum_{\substack{R \in \mathcal{D}_{good} \\ \ell(R) \leq 2^{s-r-1}}} \sigma(R) \bigg( \sum_{k=r+1}^{s-\log_2 \ell(R)}2^{-\alpha k/2} \sigma(R^{(k-1)})^{-1/2} \big\|
\Delta_{R^{(k)}} f \big\|_{L^2(\sigma)} \bigg)^2 \\
&\lesssim \sum_{\substack{R \in \mathcal{D}_{good} \\ \ell(R) \leq 2^{s-r-1}}} \sigma(R) \sum_{k=r+1}^{s-\log_2\ell(R)}
2^{-\alpha k/2} \sigma(R^{(k-1)})^{-1} \big\| \Delta_{R^{(k)}} f \big\|_{L^2(\sigma)}^2,
\end{align*}
where $s=\log_2 \ell(P_0)$. Reindexing the above sum gives that
\begin{align*}
\Sigma_{in}^{'}
&\lesssim \sum_{k=r+1}^\infty 2^{- \alpha k/2} \sum_{j=k-s}^\infty \sum_{S:\ell(S)=2^{k-j-1}} \big\| \Delta_{S^{(1)}} f \big\|_{L^2(\sigma)}^2 \sigma(S)^{-1}
\sum_{\substack{R:\ell(R)=2^{-j} \\ R \subset S}} \sigma(R) \\
&=\sum_{k=r+1}^\infty 2^{- \alpha k/2}\sum_{j=k-s}^\infty \sum_{S:\ell(S)=2^{k-j-1}}\big\| \Delta_{S^{(1)}}f \big\|_{L^2(\sigma)}^2 \\
&\lesssim \sum_{k=r+1}^\infty 2^{-\alpha k/2}\sum_{j=k-s}^\infty \sum_{S:\ell(S)=2^{k-j}}\big\| \Delta_{S} f \big\|_{L^2(\sigma)}^2 \\
&\lesssim \sum_{S:\ell(S) \leq 2^s} \big\| \Delta_{S} f \big\|_{L^2(\sigma)}^2
\lesssim \big\| f \big\|_{L^2(\sigma)}^2.
\end{align*}
\noindent\textbf{$\bullet$ Estimate of $\Sigma_{in}^{''}$.}
The following lemma will be needed in the estimate of $\Sigma_{in}^{''}$.
\begin{lemma}
Let $0 < \alpha \leq m(\lambda - 2)/2$. For each cube $S \in ch(R^{(k)})$ and $S \neq R^{(k-1)}$, it holds that
\begin{align*}
\bigg(\iint_{W_R} \int_{\Rn} |\widetilde{\theta}_t^{\sigma}(\mathbf{1}_{S} \Delta_{R^{(k)}} f)(y)|^2 &\Big(\frac{t}{t+|x-y|}\Big)^{m\lambda}
\frac{d\mu(y)}{t^m} d\sigma(x) \frac{dt}{t} \bigg)^{1/2} \\
&\lesssim  \ 2^{- \alpha k/2}\sigma(R)^{1/2} \sigma({R^{(k-1)}})^{-1/2} \big\| \Delta_{R^{(k)}} f \big\|_{L^2(\sigma)}.
\end{align*}
\end{lemma}

\begin{proof}
It suffices to prove that
\begin{align*}
\bigg( \int_{\Rn} |\widetilde{\theta}_t^{\sigma}(\mathbf{1}_{S} \Delta_{R^{(k)}} f)(y)|^2 \Big(\frac{t}{t+|x-y|}\Big)^{m\lambda}
\frac{d\mu(y)}{t^m} \bigg)^{1/2}
\lesssim 2^{- \alpha k/2} \sigma({R^{(k-1)}})^{-1/2} \big\| \Delta_{R^{(k)}} f \big\|_{L^2(\sigma)}.
\end{align*}
Indeed, combining the techniques in Lemma $\ref{estimate-1}$ with the arguments in $(\ref{R-R-P-R})$, we get the desired estimate immediately.

\end{proof}

Proceeding as what we have done in the estimate of $\Sigma_{in}^{'}$, we will deduce that
$
\Sigma_{in}^{''} \lesssim \big\| f \big\|_{L^2(\sigma)}^2.
$

\vspace{0.3cm}
\noindent\textbf{$\bullet$ Paraproduct estimate.}
Our aim here is to bounded the last term $\Sigma_{in}^{'''}$. We begin by showing that $\{ a_P \}_{P \in \mathcal{D}}$ is a Carleson sequence, where
$$
a_P = \sum_{\substack{R \in \mathcal{D}_0^{tr}:R \subset P_0 \\ R:\mathcal{D}-{good} \\ \ell(R)<2^{-r} \ell(P_0) \\ P=R^{(r)}}} \int_{R}
\int_{\ell(R)/2}^{\min\{\ell(R),\ell(Q)\}}\int_{\Rn}|\widetilde{\theta}_t^\sigma b(y)|^2 \Big(\frac{t}{t+|x-y|}\Big)^{m\lambda}  \frac{d\mu(y)}{t^m}\frac{dt}{t} d\sigma(x).
$$
Actually, it holds that
\begin{align*}
\sum_{\substack{P \in \mathcal{D}\\ P \subset S}} a_P
&\leq \sum_{\substack{R \in \mathcal{D}_0^{tr} \\ R \subset S}}
\iint_{[S \times (0,\ell(Q))] \cap W_R} \int_{\Rn}|\widetilde{\theta}_t^{\sigma} b(y)|^2 \Big(\frac{t}{t+|x-y|}\Big)^{m\lambda}  \frac{d\mu(y)}{t^m} d\sigma(x) \frac{dt}{t} \\
&\leq \iint_{S \times (0,\ell(Q))} \int_{\Rn}|\widetilde{\theta}_t^{\sigma} b(y)|^2 \Big(\frac{t}{t+|x-y|}\Big)^{m\lambda}  \frac{d\mu(y)}{t^m} d\sigma(x) \frac{dt}{t} \\
&:=\int_{S} \widetilde{g}_{\lambda,\sigma,Q}^{*}(b)(x) d\sigma(x)
\lesssim \sigma(S).
\end{align*}
It is necessary and worth pointing out that
$\widetilde{g}_{\lambda,\sigma,Q}^{*}(b)(x)=g_{\lambda,\sigma,Q}^{*}(b)(x)\mathbf{1}_{\Rn \setminus S_0}(x) \leq \xi_0$ and
$S_0$ is defined at the beginning of subsection $\ref{subsec-first}$.

It remains only to consider the bound of $\Sigma_{in}^{'''}$. The condition $R^{(r)} \not\subset H \cup T$ implies that
$|\langle b \rangle_{R^{(r)}}| \gtrsim 1$. Consequently, by the Carleson embedding theorem, it yields that
\begin{align*}
\mathcal{J}_3
&\lesssim \sum_{\substack{R \in \mathcal{D}_{good} \\ \ell(R) \leq 2^{s-r-1}}} |\langle f \rangle_{R^{(r)}}|^2 \iint_{W_R} \int_{\Rn} |\widetilde{\theta}_t^\sigma b(y)|^2
\Big(\frac{t}{t+|x-y|}\Big)^{m \lambda}\frac{d\mu(y)}{t^m} d\sigma(x) \frac{dt}{t}\\
&\leq \sum_{S \in \mathcal{D}} |\langle f \rangle_{S}|^2 \sum_{\substack{R \in \mathcal{D}_{good}\\ S=R^{(r)}}}
\iint_{W_R} \int_{\Rn} |\widetilde{\theta}_t^\sigma b(y)|^2 \Big(\frac{t}{t+|x-y|}\Big)^{m \lambda}\frac{d\mu(y)}{t^m} d\sigma(x) \frac{dt}{t}\\
&\lesssim  \sum_{S \in \mathcal{D}} |\langle f \rangle_{S}|^2 a_S
\lesssim \big\| f \big\|_{L^2(\sigma)}^2.
\end{align*}
Therefore, in all, we have completed the proof of Theorem $\ref{Theorem-big}$.

\qed
\section{Non-homogeneous local $Tb$ theorem}\label{Sec-local}
This section will be devoted to demonstrate the local $Tb$ Theorem $\ref{Theorem-Local}$. The Proposition below will be used in the proof.
\begin{proposition}\label{Pro-B1-B2}
Let $\mu$ be an upper power bound measure and $B_1,B_2 < \infty$, $\epsilon_0 \in (0,1)$ be given constants. Let $Q \subset \Rn$ be a fixed cube. Assume that there exists a complex
measure $\nu=\nu_Q$ such that
\begin{enumerate}
\item [(1)] $\supp \nu \subset Q$;
\item [(2)] $\mu(Q)=\nu(Q)$;
\item [(3)] $||\nu|| \leq B_1 \mu(Q)$;
\item [(4)] For all Borel sets $A \subset Q$ satisfying $\mu(A) \leq \epsilon_0 \mu(Q)$, we have
$|\nu|(A) \leq \frac{||\nu||}{32B_1}$.
\end{enumerate}
Suppose that there exists $s>0$ and a Borel set $U_Q \subset \Rn$ for which $|\nu|(U_Q) \leq \frac{||\nu||}{16B_1}$ so that
$$
\sup_{\zeta >0} \zeta^s \mu\big(\{x \in Q \setminus U_Q; g_{\lambda,Q}^*(\nu)(x) > \zeta\}\big) \leq B_2 ||\nu||.
$$
Then there is some subset $G_Q \subset Q \setminus U_Q$, such that, for any $f \in L^2(\mu)$ with $\supp f \subset G_Q$, it holds that
\begin{enumerate}
\item [(i)] $\mu(G_Q) \simeq \mu(Q)$;
\item [(ii)] $\big\| \mathbf{1}_{G_Q} g_{\lambda,\mu}^*(f) \big\|_{L^2(\mu)} \lesssim \big\| f \big\|_{L^2(\mu)}$.
\end{enumerate}
\end{proposition}
\vspace{0.3cm}
\noindent\textbf{Proof of Theorem $\ref{Theorem-Local}$.}
By Proposition $\ref{Pro-B1-B2}$, for every $(2,\beta)$-doubling cube $Q \subset \Rn$ with $c_1$-small boundary satisfying the assumptions in Theorem
$\ref{Theorem-Local}$, there exists a subset $G_Q \subset Q$ such that
$$
\mu(G_Q) \simeq \mu(Q) \ \ \text{and} \ \
\big\| \mathbf{1}_{G_Q} g_{\lambda,\mu}^*(f) \big\|_{L^2(\mu)} \lesssim \big\| f \big\|_{L^2(\mu)},
$$
for any $f \in L^2(\mu)$ with $\supp f \subset G_Q$. Therefore, by Proposition $\ref{Pro-2-M}$ in next section, it follows that
\begin{equation}\label{4.1}
g_{\lambda}^* : \mathfrak{M}(\Rn) \rightarrow L^{1,\infty}(\mu \lfloor G_Q).
\end{equation}
Thus, the $L^p(\mu)$ boundedness of $g_{\lambda,\mu}^*$ follows from \ref{4.1} and Theorem $\ref{Theorem-L^p}$.

\qed

The remainder of this section is devoted to give the proof of Proposition $\ref{Pro-B1-B2}$. In this step, the big piece $Tb$ Theorem $\ref{Theorem-big}$ will be used.

\vspace{0.3cm}
\noindent\textbf{Proof of Proposition $\ref{Pro-B1-B2}$.}
To apply Theorem $\ref{Theorem-big}$, it is essential to construct the exceptional set $H$. Some ideas will be taken from \cite{MMV}, which essentially goes back to
\cite[~p.138]{Tolsa}. For the sake of descriptive integrality, we give the details of the construction.

We may assume that $\supp \mu \subset Q$. Denote $\sigma=|\nu|$. Using the decomposition theorem of complex measure, we find a Borel measurable function $b$ such that $d\nu=b d\sigma$
and $|b(x)| \equiv 1$.
Set
$$
T_w = \bigcup_{R \in \mathcal{A}_w}R,\ \ \text{and} \
\mathcal{A}_w = \big\{\text{maximal} \ R \in \mathcal{D}(w);|\langle b \rangle_R^{\sigma}| < \eta) \big\},
$$
where $\eta=1/(2B_1)$. The assumptions give that
\begin{align*}
B_1^{-1} \sigma(Q) = B_1^{-1} ||\nu|| \leq  \mu(Q)
&= \nu(Q) = \bigg| \int_Q b \ d\sigma \bigg|
=\bigg| \int_{Q \setminus T_w} b \ d\sigma + \sum_{R \in \mathcal{A}_w} \int_R b \ d\sigma \bigg| \\
&\leq \sigma(Q \setminus T_w) + \eta \sigma(Q)
=(1+\eta)\sigma(Q) - \sigma(T_w).
\end{align*}
Then, one obtains that $\sigma(T_w) \leq (1-\eta)\sigma(Q)$.

Let $p(x)=\sup_{r>0} r^{-m} \sigma(B(x,r))$ {and} $ E_{p_0}=\{x \in \Rn; p(x) \geq p_0\} $ {for} $\ p_0 > 0.
$ By the fact $p: \mathfrak{M}(\Rn) \rightarrow L^{1,\infty}(\mu)$, it yields that
$$
\mu(E_{p_0})
\leq C p_0^{-1} ||\nu|| \leq C B_1 p_0^{-1} \mu(Q) \leq \epsilon_0 \mu(Q),
$$
whenever $p_0$ is large enough. Then the condition $(4)$ implies that $\sigma(E_{p_0/2^m}) \leq \eta/8 \sigma(Q)$.
For any $x$ with $p(x)>p_0$, set
$$
r(x)=\sup \{r>0;\sigma(B(x,r)) > p_0 r^m \},\ \ H_1 := \bigcup_{x:p(x)>p_0}B(x,r(x)).
$$
It is easy to see that the ball $B_r$ with $\sigma(B_r) > p_0 r^m$ satisfies $B_r \subset H_1$, and $H_1 \subset E_{p_0/2^m}$.
Hence, $\sigma(H_1) \leq \eta/8 \sigma(Q)$.

Now, we introduce the notions $\mathcal{F}_1$ and $\mathcal{F}_2$ as follows:
\begin{eqnarray*}
\mathcal{F}_1&=\big\{\text{maximal}\ R \in \D_0;\ \sigma(R) > \epsilon_0^{-1} B_1 \mu(R) \big\},
\mathcal{F}_2&=\big\{\text{maximal}\ R \in \D_0;\ \sigma(R) < \delta\mu(R)\big\},
\end{eqnarray*}
where $\delta=1/(32 B_1)$. Then we obtain that
$$
\mu \Big( \bigcup_{R \in \mathcal{F}_1} R \Big) \leq \epsilon_0 \mu(Q),\
\sigma \Big( \bigcup_{R \in \mathcal{F}_1} R \Big) \leq \delta \sigma(Q).
$$
In addition, we can further deduce that
$$
\sigma \Big( \bigcup_{R \in \mathcal{F}_2} R \Big)
\leq \delta \sum_{R \in \mathcal{F}_2} \mu(R)
\leq \mu(Q) \leq \delta \sigma(Q).
$$
By setting $H_2 = \bigcup_{R \in \mathcal{F}_1 \cup \mathcal{F}_2} R$, we have $\sigma(H_2) \leq 2 \delta \sigma(Q)$.
Below, we will discuss what kinds of properties does $H_2$ enjoys. If $x \in Q \setminus H_2$, then for each $R \in \mathcal{D}_0$ contains $x$, we have
$$
\delta \leq \frac{\sigma(R)}{\mu(R)} \leq \frac{B_1}{\epsilon_0}.
$$
This implies that $\sigma \lfloor (Q \setminus H_2) \ll  \mu \lfloor (Q \setminus H_2)$.
By Radon-Nikodym theorem, for all Borel sets $A \subset Q \setminus H_2$, one can find a function $\varphi \geq 0$ such that
$\sigma(A)=\int_A \varphi \ d\mu$. Moreover,
$\varphi \simeq 1$ for $\mu$-a.e. $x \in Q \setminus H_2$.

Let $H=H_1 \cup H_2 \cup U_Q$. Then $H$ enjoys the following properties :
\begin{enumerate}
\item [(1)] $\sigma(H \cup T_w) \leq (1-\eta/2) \sigma(Q)$;
\item [(2)] If $\sigma(B_r) > p_0 r^m$, then $B_r \subset H$;
\item [(3)] For each Borel set $A \subset Q \setminus H$, the function $\varphi$ satisfies
$$
\varphi(x) \simeq 1 \ \mu-a.e. \ x\in Q \setminus H \ \ \ \text{and} \  \
\sigma(A)=(\varphi d\mu)(A).
$$
\item [(4)] For any $\zeta > 0$, it holds that \begin{align*}
&\zeta^s \sigma \big(\{x \in Q \setminus H; g_{\lambda,\sigma,Q}^*(b)(x) > \zeta \}\big)
=\zeta^s \sigma \big(\{x \in Q \setminus H; g_{\lambda,Q}^*(\nu)(x) > \zeta \}\big) \\
&\lesssim \zeta^s \mu \big(\{x \in Q \setminus U_Q; g_{\lambda,\sigma,Q}^*(b)(x) > \zeta \}\big)
\leq B_2 ||\nu|| = B_2 \sigma(Q).
\end{align*}
\end{enumerate}
Using big piece $Tb$ Theorem $\ref{Theorem-big}$, for any $f \in L^2(\sigma)$, one can find $G_Q \subset Q \setminus H \subset Q \setminus U_Q$ such that
\begin{equation}\label{GQ-Q}
\sigma(G_Q) \simeq \sigma(Q) \ \ \text{and} \ \
|| \mathbf{1}_{G_Q} g_{\lambda,\sigma,Q}^*(f) ||_{L^2(\sigma)} \lesssim ||f||_{L^2(\sigma)}.
\end{equation}

Suppose that $h \in L^2(\mu)$ and $\supp h \subset G_Q$. The fact $G_Q \subset Q \setminus H$ implies that
$\varphi \simeq 1$ $\mu$-a.e. on the support of $h$. From $(\ref{GQ-Q})$ with respect to $f=h/\varphi$, it follows that
\begin{align*}
|| \mathbf{1}_{G_Q} g_{\lambda,\mu,Q}^*(h) ||_{L^2(\mu)}
\lesssim || \mathbf{1}_{G_Q} g_{\lambda,\mu,Q}^*(h) ||_{L^2(\sigma)}
&=|| \mathbf{1}_{G_Q} g_{\lambda,\sigma,Q}^*(h/\varphi) ||_{L^2(\sigma)} \lesssim ||h/\varphi||_{L^2(\sigma)}
\lesssim ||h||_{L^2(\mu)}.
\end{align*}
Applying the size condition and the H\"{o}lder inequality, we obtain
$$
|\theta_t^\mu h(y)| \lesssim t^{-m} \mu(G_Q)^{1/2} ||h||_{L^2(\mu)}.
$$
Consequently, it yields that
\begin{align*}
\mathfrak{G}&:=\bigg(\int_{G_Q}\int_{\ell(Q)}^{\infty}\int_{\Rn} \Big(\frac{t}{t+|x-y|}\Big)^{m\lambda} |\theta_t^\mu h(y)|^2 \frac{d\mu(y) dt}{t^{m+1}d\mu(x)}\bigg)^{1/2} \\
&\lesssim \mu(G_Q)^{1/2} ||h||_{L^2(\mu)}\bigg(\int_{G_Q}\int_{\ell(Q)}^{\infty}\frac{1}{t^{2m}}\frac{dt}{t} d\mu(x)\bigg)^{1/2}
\bigg(\int_{\Rn}\Big(\frac{t}{t+|x-y|}\Big)^{m\lambda} \frac{d\mu(y)}{t^m}\bigg)^{1/2} \\
&\lesssim \ell(Q)^{-m} \mu(G_Q) ||h||_{L^2(\mu)}
\lesssim ||h||_{L^2(\mu)}.
\end{align*}
Hence, for any $h \in L^2(\mu)$ with $\supp h \subset G_Q$, we have
$$
|| \mathbf{1}_{G_Q} g_{\lambda,\mu}^*(h) ||_{L^2(\mu)}
\leq || \mathbf{1}_{G_Q} g_{\lambda,\mu,Q}^*(h) ||_{L^2(\mu)} + \mathfrak{G}
\lesssim ||h||_{L^2(\mu)},
$$ Additionally, there holds that
$$
\mu(Q) \leq \sigma(Q) \simeq \sigma(G_Q)=\int_{G_Q} \varphi d\mu \lesssim \mu(G_Q).
$$
This finishes the proof of Proposition $\ref{Pro-B1-B2}$..

\qed
\section{The $\mathfrak{M}(\Rn) \rightarrow L^{1,\infty}(\mu)$ Bound}\label{Sec-M}
The following proposition has been used to demonstrate Theorem $\ref{Theorem-Local}$ in the above section. Here we present its proof.
\begin{proposition}\label{Pro-2-M}
Let $\lambda > 2$, $0 < \alpha \leq m(\lambda-2)/2$ and $\mu$ be a power bound measure. Suppose that $g_{\lambda,\mu}^*$ is bounded on $L^2(\mu)$, then
$g_{\lambda}^*$ is bounded from $\mathfrak{M}(\Rn)$ to $L^{1,\infty}(\mu)$.

In particular, the $L^2(\mu)$ boundedness of $g_{\lambda,\mu}^*$ implies its weak $(1,1)$ boundedness.
\end{proposition}
In order to prove that $g_{\lambda}^*$ is bounded from $\mathfrak{M}(\Rn)$ to $L^{1,\infty}(\mu) $, we need a substitute for the Calder\'{o}n-Zygmund decomposition of a measure
\cite{Tolsa}, suitable for non-doubling measures.
\begin{lemma}\label{C-Z decomposition}
Let $\mu$ be a Radon measure on $\Rn$. For any $\nu \in \mathfrak{M}(\Rn)$ with compact support and any
$\xi > 2^{n+1} ||\nu||/||\mu||$, we have:
\begin{enumerate}
\item [(a)] There exists a family of almost disjoint cubes $\{Q_i\}_i$ and a function $f \in L^1(\mu)$ such that
\begin{eqnarray}
\label{C-Z-1}|\nu|(Q_i) & > & \frac{\xi}{2^{n+1}} \mu(2Q_i);\\
\label{C-Z-2}|\nu|(\eta Q_i) & \leq & \frac{\xi}{2^{n+1}} \mu(2\eta Q_i), \ for \ any \ \eta > 2;\\
\label{C-Z-3}\nu = f \mu \ \ & {}& in \ \Rn \setminus \bigcup_i Q_i,\ \ with \ |f| \leq  \xi \ \ \mu-a.e.;
\end{eqnarray}
\item [(b)] For each $i$, let $R_i$ be a $(6, 6^{m+1})$-doubling cube concentric with $Q_i$, with $\ell(R_i) > 4\ell(Q_i)$ and denote $w_i = \mathbf{1}_{Q_i} \big/
    {\sum_{k}\mathbf{1}_{Q_k}}$. Then, there exists a family of functions $\varphi_i$ with $supp(\varphi_i) \subset R_i$, and each $\varphi_i$ with constant sign satisfying
\begin{eqnarray}
\label{C-Z-4}\int_{R_i} \varphi_i \ d\mu & = & \int_{Q_i} f w_i \ d\mu;\\
\label{C-Z-5}\sum_{i} |\varphi_i| & \lesssim & \xi;\\
\label{C-Z-6}\mu(R_i) ||\varphi_i||_{L^\infty(\mu)} & \lesssim & |\nu|(Q_i).
\end{eqnarray}
\end{enumerate}
\end{lemma}
\qed

\noindent\textbf{Proof of Proposition \ref{Pro-2-M}}. For simplicity we may assume that $\nu$ has compact support and $\xi > 2^{n+1}||\nu|| / ||\mu||$. Applying Lemma $\ref{C-Z decomposition}$, we have the decomposition
$\nu=g \mu + \beta$, where
$$
g \mu = \mathbf{1}_{\Rn \setminus \bigcup Q_i} \nu + \sum_i \varphi_i \mu,\ \
\beta = \sum_i \beta_i := \sum_i ( w_i \nu - \varphi_i \mu).
$$
Then we have
\begin{align*}
\mu\big(\{x \in \Rn; g_\lambda^* \nu(x) > \xi \}\big)
&\leq \mu\Big(\bigcup_i 2 Q_i \Big) + \mu \Big(\Big\{x \in \Rn \setminus \bigcup_i 2Q_i; g_{\lambda,\mu}^*(g)(x) > \xi/2 \Big\}\Big) \\
&\quad +\mu\Big(\Big\{x \in \Rn \setminus \bigcup_i 2Q_i; g_\lambda^*(\beta)(x) > \xi/2 \Big\}\Big)\\
&:= \mathcal{I}_{c} + \mathcal{I}_{g} + \mathcal{I}_{b}.
\end{align*}
For the first part $\mathcal{I}_{c}$ , by $(\ref{C-Z-1})$ we have
$$
\mathcal{I}_{c} \lesssim \frac{1}{\xi}\sum_{i} |\nu|(Q_i) \leq \frac{1}{\xi}||\nu||.
$$
For the good part $\mathcal{I}_{g}$, one obtains that
\begin{align*}
\mathcal{I}_{g}
&\lesssim \frac{1}{\xi^2} \int_{\Rn \setminus \bigcup_i Q_i}g_{\lambda,\mu}^*(g)(x)^2 d\mu(x)
\lesssim \frac{1}{\xi^2} \int_{\Rn \setminus \bigcup_i Q_i}|g(x)|^2 d\mu(x)
\lesssim \frac{1}{\xi} \int_{\Rn \setminus \bigcup_i Q_i}|g(x)| d\mu(x) \\
&\leq \frac{1}{\xi} |\nu|\Big(\Rn \setminus \bigcup_i Q_i\Big) + \frac{1}{\xi} \int_{\Rn} \Big|\sum_i \varphi_i(x)\Big| d\mu(x)\leq \frac{1}{\xi} ||\nu|| + \frac{1}{\xi} \sum_i \int_{R_i} |\varphi_i(x)| d\mu(x)\\&
\lesssim \frac{1}{\xi} ||\nu|| + \frac{1}{\xi} \sum_i ||\varphi_i||_{L^\infty(\mu)} \mu(R_i) \lesssim \frac{1}{\xi} ||\nu|| + \frac{1}{\xi} \sum_i |\nu|(Q_i)
\lesssim \frac{1}{\xi} ||\nu||.
\end{align*}
For the bad part $\mathcal{I}_{b}$, we may get
\begin{align*}
\mathcal{I}_{b}
&\lesssim \frac{1}{\xi} \int_{\Rn \setminus \bigcup_i 2 Q_i} g_\lambda^*(\beta)(x) d\mu(x)
\lesssim \frac{1}{\xi} \sum_i \int_{\Rn \setminus 2 Q_i} g_\lambda^*(\beta_i)(x) d\mu(x) \\
&\lesssim \frac{1}{\xi} \sum_i \int_{\Rn \setminus \bigcup_i 4 R_i} g_\lambda^*(\beta_i)(x) d\mu(x)
+ \frac{1}{\xi} \sum_i \int_{4R_i \setminus \bigcup_i 2Q_i} g_\lambda^*(w_i \nu)(x) d\mu(x) \\
&\quad +  \frac{1}{\xi} \sum_i \int_{4R_i \setminus 2 Q_i} g_{\lambda,\mu}^*(\varphi_i)(x) d\mu(x):= \frac{1}{\xi}(\mathcal{I}_{b,1} + \mathcal{I}_{b,2} + \mathcal{I}_{b,3}).
\end{align*}
Note that we may bound $\mathcal{I}_{b,3}$ in the following way:
\begin{align*}
\mathcal{I}_{b,3}
&\leq \sum_i \int_{4R_i} g_{\lambda,\mu}^*(\varphi_i)(x) d\mu(x)
\leq \sum_i \mu(4R_i)^{1/2} \bigg( \int_{4R_i} g_{\lambda,\mu}^*(\varphi_i)(x)^2 d\mu(x)\bigg)^{1/2} \\
&\lesssim \sum_i \mu(R_i)^{1/2} \bigg( \int_{R_i} |\varphi_i(x)|^2 d\mu(x)\bigg)^{1/2}
\lesssim \sum_i \mu(R_i)||\varphi_i||_{L^\infty(\mu)} \\
&\lesssim \sum_i |\nu|(Q_i)
\lesssim ||\nu||.
\end{align*}
Therefore, in order to obtain the weak $(1,1)$ type bound, we only need to show
$$\mathcal{I}_{b,i} \lesssim ||\nu||,\ i=1,2.$$
Thus, it is enough to prove that for each $i$ it holds that
\begin{equation}\label{bad-1}
\int_{\Rn \setminus 4R_i} g_\lambda^*(\beta_i)(x) d\mu(x) \lesssim |\nu|(Q_i)
\end{equation}
and
\begin{equation}\label{bad-2}
\int_{4R_i \setminus 2Q_i} g_\lambda^*(w_i \nu)(x) d\mu(x) \lesssim |\nu|(Q_i).
\end{equation}
\vspace{0.3cm}
\noindent\textbf{$\bullet$ Estimating bad part $\mathcal{I}_{b,1}$.}
To gain the inequality $(\ref{bad-1})$, since $\big\| \beta_i \big\| \lesssim |\nu|(Q_i)$, it suffices to prove that
\begin{equation}\label{bad-3}
g_\lambda^*(\beta_i)(x) \lesssim \bigg( \frac{\ell(R_i)^\alpha}{|x-c_{R_i}|^{m+\alpha}} + \frac{\ell(R_i)^{\alpha/2}}{|x-c_{R_i}|^{m+\alpha/2}} \bigg) \big\| \beta_i \big\|, \ x \in
\Rn \setminus 4R_i.
\end{equation}
\begin{proof}
The first step is to split
\begin{align*}
g_\lambda^*(\beta_i)(x)^2
&\lesssim \iint_{\Rn \times (0,\ell(R_i))} \Big(\frac{t}{t+|x-y|}\Big)^{m\lambda} |\theta_t \beta_i(y)|^2 \frac{d\mu(y) dt}{t^{m+1}} \\
&\quad + \iint_{\Rn \times [\ell(R_i),|x-c_{R_i}|]} \Big(\frac{t}{t+|x-y|}\Big)^{m\lambda} |\theta_t \beta_i(y)|^2 \frac{d\mu(y) dt}{t^{m+1}} \\
&\quad + \iint_{\Rn \times (|x-c_{R_i}|,+\infty)} \Big(\frac{t}{t+|x-y|}\Big)^{m\lambda} |\theta_t \beta_i(y)|^2 \frac{d\mu(y) dt}{t^{m+1}}\\
&:= \mathcal{A}_1(x) + \mathcal{A}_2(x) + \mathcal{A}_3(x).
\end{align*}
For $\mathcal{A}_3$, by the vanishing property $\beta_i(R_i)=0$ and H\"{o}lder condition, we gain that
\begin{align*}
|\theta_t \beta_i(y)| &= \bigg| \int_{R_i} (s_t(y,z)-s_t(y,c_{R_i})) d\beta_i(z) \bigg| \\
&\lesssim \int_{R_i} \frac{|z-c_{R_i}|^\alpha}{(t+|y-z|)^{m+\alpha}} d|\beta_i|(z)
\lesssim \frac{\ell(R_i)^\alpha}{t^{m+\alpha}} || \beta_i ||.
\end{align*}
Furthermore, we have
\begin{align*}
\mathcal{A}_3(x)
&\lesssim || \beta_i || \bigg(\int_{|x-c_{R_i}|}^{\infty} \frac{\ell(R_i)^{2\alpha}}{t^{2m+2\alpha+1}} \int_{\Rn}\Big(\frac{t}{t+|x-y|}\Big)^{m\lambda} \frac{d\mu(y)}{t^m} dt
\bigg)^{1/2} \\
&\lesssim || \beta_i || \bigg(\int_{|x-c_{R_i}|}^{\infty} \frac{\ell(R_i)^{2\alpha}}{t^{2m+2\alpha+1}}dt \bigg)^{1/2}
\simeq || \beta_i || \frac{\ell(R_i)^\alpha}{|x-c_{R_i}|^{m+\alpha}}.
\end{align*}

We then consider $\mathcal{A}_1(x)$. If $y \in \Rn$ satisfies $|y-c_{R_i}| \leq \frac{1}{2}|x-c_{R_i}|$, then
$|x-y| \geq |x-c_{R_i}|-|y-c_{R_i}| \geq \frac12 |x-c_{R_i}|$. Hence,
\begin{align*}
&\int_{y: |y-c_{R_i}| \leq \frac{1}{2}|x-c_{R_i}|} |\theta_t \beta_i(y)|^2 \Big(\frac{t}{t+|x-y|}\Big)^{m\lambda} \frac{d\mu(y) }{t^m} \\
&\lesssim \int_{y: |y-c_{R_i}| \leq \frac{1}{2}|x-c_{R_i}|}\bigg(\int_{R_i}\frac{t^\alpha}{(t+|y-z|)^{m+\alpha}} d|\beta_i|(z)
\bigg)^2 \Big(\frac{t}{t+|x-y|}\Big)^{m\lambda} \frac{d\mu(y)}{t^m} \\
&\lesssim \frac{t^{m\lambda - 2m}}{(t+|x-c_{R_i}|)^{m\lambda}} t^{-m} \int_{\Rn} \bigg(\int_{R_i} \Big(\frac{t}{t+|x-y|}\Big)^{m+\alpha} d|\beta_i|(z)\bigg)^2 d\mu(y) \\
&\lesssim \frac{t^{2\alpha}}{(t+|x-c_{R_i}|)^{2m+2\alpha}} t^{-m} \Big\| \Big(\frac{t}{t+|\cdot|}\Big)^{m+\alpha} \Big\|_{L^2(\mu)}^2 \big\| \beta_i \big\|^2 \\
&\lesssim \frac{t^{2\alpha}}{|x-c_{R_i}|^{2m+2\alpha}} || \beta_i ||^2.
\end{align*}
If $y,z\in \Rn$ satisfies $|y-c_{R_i}| > \frac{1}{2}|x-c_{R_i}|$ and $z \in R_i$, then
$$|z-c_{R_i}| \leq \frac12 \ell(R_i) \leq \frac12 \cdot \frac23 |x-c_{R_i}| = \frac13 |x-c_{R_i}|$$
and
$$|y-z| \geq |y-c_{R_i}| - |z-c_{R_i}| > \frac12 |x-c_{R_i}| - \frac13 |x-c_{R_i}| = \frac16 |x-c_{R_i}|.$$
Therefore, we have
\begin{align*}
&\int_{y: |y-c_{R_i}| > \frac{1}{2}|x-c_{R_i}|} |\theta_t \beta_i(y)|^2 \Big(\frac{t}{t+|x-y|}\Big)^{m\lambda} \frac{d\mu(y) }{t^m} \\
&\lesssim \int_{y: |y-c_{R_i}| > \frac{1}{2}|x-c_{R_i}|}\bigg(\int_{R_i}\frac{t^\alpha}{(t+|y-z|)^{m+\alpha}}d|\beta_i|(z)
\bigg)^2 \Big(\frac{t}{t+|x-y|}\Big)^{m\lambda} \frac{d\mu(y)}{t^m} \\
&\lesssim \frac{t^{2\alpha}}{|x-c_{R_i}|^{2m+2\alpha}} || \beta_i ||^2 \int_{\Rn} \Big(\frac{t}{t+|x-y|}\Big)^{m+\alpha} \frac{d\mu(y)}{t^m}\lesssim \frac{t^{2\alpha}}{|x-c_{R_i}|^{2m+2\alpha}} || \beta_i ||^2.
\end{align*}

Finally, we consider the estimate for $\mathcal{A}_2(x)$. It follows from H\"{o}lder condition that
\begin{align*}
|\theta_t \beta_i(y)| &= \bigg| \int_{R_i} (s_t(y,z)-s_t(y,c_{R_i})) d\beta_i(z) \bigg|
\lesssim \int_{R_i} \frac{|z-c_{R_i}|^\alpha}{(t+|y-z|)^{m+\alpha}} d|\beta_i|(z) \\
&\lesssim \ell(R_i)^{\alpha/2} \int_{R_i} \frac{t^{\alpha/2}}{(t+|y-z|)^{m+\alpha}} d|\beta_i|(z).
\end{align*}
This gives us that
\begin{align*}
&\int_{\Rn} |\theta_t \beta_i(y)|^2 \Big(\frac{t}{t+|x-y|}\Big)^{m\lambda} \frac{d\mu(y)}{t^m}
\lesssim \ell(R_i)^\alpha \frac{t^\alpha}{|x-c_{R_i}|^{2m+2\alpha}} || \beta_i ||^2.
\end{align*}
Therefore,
\begin{align*}
\mathcal{A}_2(x)
\lesssim  \bigg( \int_{\ell(R_i)}^{|x-c_{R_i}|} \frac{\ell(R_i)^\alpha}{|x-c_{R_i}|^{2m+2\alpha}} || \beta_i ||^2 t^{\alpha-1} dt \bigg)^{1/2}
\simeq \frac{\ell(R_i)^{\alpha/2}}{|x-c_{R_i}|^{m+\alpha/2}} || \beta_i ||.
\end{align*}
\end{proof}
\vspace{0.3cm}
\noindent\textbf{$\bullet$ Estimating bad part $\mathcal{I}_{b,2}$.}
In order to get the inequality$(\ref{bad-2})$, it is sufficient to prove the following estimate
\begin{equation}\label{bad-4}
g_\lambda^*(w_i \nu)(x) \lesssim \frac{|\nu|(Q_i)}{|x-c_{Q_i}|^m} , \ x \in 4R_i \setminus 2Q_i.
\end{equation}
Actually,
\begin{align*}
\int_{4R_i \setminus 2Q_i} \frac{d\mu(x)}{|x-c_{Q_i}|^m}
&\leq \bigg(\int_{4R_i \setminus R_i} +  \int_{R_i \setminus 6Q_i} + \int_{6 Q_i \setminus Q_i}\bigg) \frac{d\mu(x)}{|x-c_{Q_i}|^m}.
\end{align*}
It is easy to see that
$$\bigg(\int_{4R_i \setminus R_i} + \int_{6 Q_i \setminus Q_i}\bigg) \frac{d\mu(x)}{|x-c_{Q_i}|^m}
\lesssim \frac{\mu(4R_i)}{\ell(R_i)^m} + \frac{\mu(6Q_i)}{\ell(Q_i)^m} \lesssim 1.$$
Moreover, there are no $(6,6^{m+1})$-doubling cubes of the form $6^kQ_i$ such that $6Q_i \subsetneq 6^k Q_i \subsetneq R_i$.
Let $ N_i := \min\{ k; R_i \subset 6^k \cdot 6Q_i \}$.
Hence, $$ \mu(6 \cdot 6^k Q_i) > 6^{m+1} \mu(6^k Q_i), \ k=1,\ldots,N_i.$$
Then,
$$ \mu(6^{N_i}\cdot6Q_i) > 6^{(m+1)(N_i-k)} \mu(6^k Q_i).$$
Therefore, we have
\begin{align*}
\int_{R_i \setminus 6Q_i} \frac{d\mu(x)}{|x-c_{Q_i}|^m}
\leq \sum_{k=1}^{N_i} \int_{6^{k+1}Q_i \setminus 6^k Q_i} \frac{d\mu(x)}{|x-c_{Q_i}|^m}
\lesssim \sum_{k=1}^{N_i} \frac{\mu(6^{k+1}Q_i)}{\ell(6^k Q_i)^m}
\lesssim  \sum_{k=1}^{N_i} 6^{k-N_i} \frac{\mu(6^{N_i+1}Q_i)}{\ell(6^{N_i+1} Q_i)^m}
\lesssim 1.
\end{align*}

Now let us show that $(\ref{bad-4})$ is true. We perform the decomposition as follows:
\begin{align*}
g_\lambda^*(w_i \nu)(x)
&\leq \sum_{j=1}^3\bigg(\iint_{\Xi_j} |\theta_t (w_i \nu)(y)|^2 \Big(\frac{t}{t+|x-y|}\Big)^{m\lambda} \frac{d\mu(y) dt}{t^{m+1}}\bigg)^{1/2}
:= \sum_{j=1}^3 \mathcal{B}_j(x),\ \ x \in 4R_i \setminus 2Q_i,
\end{align*}
where
\begin{eqnarray*}
\Xi_1 &=& \Rn \times [ |x-c_{Q_i}|, +\infty),\\
\Xi_2 &=& B(c_{Q_i},0.6|x-c_{Q_i}|) \times (0,|x-c_{Q_i}|),\\
\Xi_3 &=& B(c_{Q_i},0.6|x-c_{Q_i}|)^c \times (0,|x-c_{Q_i}|).
\end{eqnarray*}

The first part $\mathcal{B}_1$ is easy. Using size condition, we get
\begin{align*}
|\theta_t (w_i \nu)(y)|
\lesssim \int_{Q_i} \frac{t^\alpha}{(t+|y-z|)^{m+\alpha}} d|\nu|(z)
\lesssim t^{-m} |\nu|(Q_i).
\end{align*}
Then, we have
\begin{align*}
\mathcal{B}_1(x)
&\lesssim |\nu|(Q_i) \bigg( \int_{|x-c_{Q_i}|}^{+\infty} \frac{1}{t^{2m+1}} \int_{\Rn} \Big(\frac{t}{t+|x-y|}\Big)^{m\lambda} \frac{d\mu(y)}{t^m} dt \bigg)^{1/2} \\
&\lesssim |\nu|(Q_i) \bigg( \int_{|x-c_{Q_i}|}^{+\infty} \frac{1}{t^{2m+1}} dt \bigg)^{1/2}
\simeq \frac{|\nu|(Q_i)}{|x-c_{Q_i}|^m}.
\end{align*}

As for $\mathcal{B}_3(x)$, we note that when $(y,t) \in \Xi_3$ and $z \in Q_i$, it holds that
$$|y-z| \geq |y-c_{Q_i}| - |z-c_{Q_i}| \geq 0.6|x-c_{Q_i}| - 0.5 |x-c_{Q_i}| \gtrsim |x-c_{Q_i}|.$$
Thus,
$$ \theta_t(w_i \nu)(y) \lesssim \frac{t^\alpha}{|x-c_{Q_i}|^{m+\alpha}} |\nu|(Q_i).$$
And hence
\begin{align*}
\mathcal{B}_3(x)
\lesssim \frac{|\nu|(Q_i)}{|x-c_{Q_i}|^{m+\alpha}} \bigg( \int_{0}^{|x-c_{Q_i}|} t^{2\alpha-1} \int_{\Rn} \Big(\frac{t}{t+|x-y|}\Big)^{m\lambda} \frac{d\mu(y)}{t^m} dt \bigg)^{1/2}
\lesssim \frac{|\nu|(Q_i)}{|x-c_{Q_i}|^m}.
\end{align*}

It remains only to estimate $\mathcal{B}_2$. For $(y,t) \in \Xi_2$,
$$ |x-y| \geq |x-c_{Q_i}| - |y-c_{Q_i}| \geq 0.4 |x-c_{Q_i}|.$$
So we have
\begin{align*}
&\int_{y: |y-c_{Q_i}| \leq \frac{2}{3}|x-c_{Q_i}|} |\theta_t (w_i \nu)(y)|^2 \Big(\frac{t}{t+|x-y|}\Big)^{m\lambda} \frac{d\mu(y)}{t^m} \\
&\lesssim \frac{t^{m\lambda - m}}{(t+|x-c_{Q_i}|)^{m\lambda}}  \int_{\Rn} \bigg(\int_{Q_i} \frac{t^\alpha}{(t+|x-y|)^{m+\alpha}} d|\nu|(z)\bigg)^2 d\mu(y) \\
&\lesssim \frac{t^{2\alpha}}{(t+|x-c_{Q_i}|)^{2m+2\alpha}} t^{-m} \Big\| \Big(\frac{t}{t+|\cdot|}\Big)^{m+\alpha} \Big\|_{L^2(\mu)}^2 |\nu|(Q_i)^2 \\
&\lesssim \frac{t^{2\alpha}}{|x-c_{Q_i}|^{2m+2\alpha}} |\nu|(Q_i)^2.
\end{align*}
Therefore, the desired result can be obtained
\begin{align*}
\mathcal{B}_2(x)
\lesssim \frac{|\nu|(Q_i)}{|x-c_{Q_i}|^{m+\alpha}} \bigg( \int_{0}^{|x-c_{Q_i}|} t^{2\alpha-1}dt \bigg)^{1/2}
\lesssim \frac{|\nu|(Q_i)}{|x-c_{Q_i}|^{m}}.
\end{align*}
Thus, we have finished the proof of Proposition $\ref{Pro-2-M}$.
\qed
\section{The $L^\infty(\mu) \rightarrow RBMO(\mu) $ Bound}\label{Sec-RBMO}
In this section, we further investigate what kind of boundedness we will obtain from $L^2(\mu)$ boundedness of $g_{\lambda,\mu}^*$.
In addition, the classical techniques used in Section $\ref{Sec-lambda}$ will be utilized again.
\begin{definition}
The regular bounded mean oscillations space $RBMO$ is defined as follows. Let $\kappa > 1$. A function $f \in L^1_{loc}(\mu)$ belongs to $RBMO(\mu)$ if there exists a constant $C$ and for any ball $B$, a constant $f_B$ (It does not demand that $f_B$ be the average $\langle f \rangle_B^{\mu} = \frac{1}{\mu(B)}\int_B f d\mu$.), such that one has
$$ \int_{B}|f-f_B| d\mu \leq C \mu(\kappa B),$$
and, whenever $B \subset B'$ are balls,
$$
|f_B - f_{B'}| \leq C \bigg( 1 + \int_{\kappa B' \setminus B} \frac{1}{|x-c_B|^m} d\mu(x)\bigg).
$$
\end{definition}
\begin{proposition}\label{Pro-RBMO}
Let $\lambda > 2$, $0 < \alpha \leq m(\lambda-2)/2$ and $\mu$ be a power bounded measure. If $g_{\lambda,\mu}^*$ is bounded on $L^2(\mu)$, then
$g_{\lambda,\mu}^*$ is bounded from $L^{\infty}(\mu)$ to $RBMO(\mu)$.
\end{proposition}
In order to prove that $g_{\lambda,\mu}^*$ is bounded from $L^\infty(\mu)$ to $RBMO(\mu)$, we need to check two conditions such that $g_{\lambda,\mu}^*(f)$ belongs to $RBMO(\mu)$. The following lemma is essential to our proof.
\begin{lemma}\label{key lemma}
Suppose that $f \in L^\infty(\mu)$ with $||f||_{L^{\infty}(\mu)} \leq 1$. Then for fixed $x_0 \in \Rn$ and $r>0$, there holds that for any $x \in B(x_0,r)$ and large enough $\kappa \geq 1$
$$
|g_{\lambda,\mu}^*(f \mathbf{1}_{\Rn \setminus B(x_0,\kappa r)})(x) - g_{\lambda,\mu}^*(f \mathbf{1}_{\Rn \setminus B(x_0,\kappa r)})(x_0)| \lesssim 1.
$$
\end{lemma}
\begin{proof}
Since $|a-b| \leq |a^2-b^2|^{1/2}$, it is enough to show
$$ |g_{\lambda,\mu}^*(f \mathbf{1}_{\Rn \setminus B(x_0,\kappa r)})(x)^2 - g_{\lambda,\mu}^*(f \mathbf{1}_{\Rn \setminus B(x_0,\kappa r)})(x_0)^2| \lesssim 1.$$
We split $\R^{n+1}_+$ into five parts:
\begin{eqnarray*}
\Theta_1 &=& \big\{(y,t);\Rn \times (0,2r] \big\}, \\
\Theta_2 &=& \big\{(y,t); |y-x| \leq t < |y-x_0|, 2r < t < \infty \big\}, \\
\Theta_3 &=& \big\{(y,t); |y-x_0| \leq t < |y-x|, 2r < t < \infty \big\}, \\
\Theta_4 &=& \big\{(y,t); \min\{|y-x_0|,|y-x| \} \geq t, 2r < t < \infty \big\},\\
\Theta_5 &=& \big\{(y,t); \max\{|y-x_0|,|y-x| \} < t, 2r < t < \infty\big\}.
\end{eqnarray*}
Making use of the similar method as for $\mathcal{A}_1$, we get
$$\int_{\Rn} |\theta_t^{\mu} (f \mathbf{1}_{\Rn \setminus B(x_0,\kappa r)})(y)|^2 \Big(\frac{t}{t+|x-y|}\Big)^{m\lambda} \frac{d\mu(y) }{t^m} \lesssim t^{2\alpha}r^{-2\alpha}, \ \text{for any} \ x \in B(x_0,r).$$
Thus, we obtain
$$\iint_{\Theta_1} |\theta_t^{\mu} (f \mathbf{1}_{\Rn \setminus B(x_0,\kappa r)})(y)|^2 \Big(\frac{t}{t+|x-y|}\Big)^{m\lambda} \frac{d\mu(y) dt}{t^{m+1}} \lesssim 1.$$
As for the remaining four terms, the proof is a routine application of the method of the proof of Lemma $\ref{good-lambda}$. We omit the details.
\end{proof}
\subsection{Verifying the first condition.}
We shall check that there exists some constant $C$, and any ball $B$, a constant $\mathcal{G}_B$, such that one has
$$ \frac{1}{\mu(\kappa B)} \int_{B} |g_{\lambda,\mu}^*(f)(x)-\mathcal{G}_B| \leq C.$$
For a given ball $B=B(x_0,r)$, we write
$ \mathcal{G}_B = g_{\lambda,\mu}^*(f \mathbf{1}_{\Rn \setminus B(x_0,\kappa r)})(x_0),$
where $\kappa$ is the same as in Lemma $\ref{key lemma}$. By means of Lemma $\ref{key lemma}$ and the $L^2(\mu)$ boundedness of $g_{\lambda,\mu}^*$, we deduce that

\begin{align*}
\frac{1}{\mu(\kappa B)} \int_{B} |g_{\lambda,\mu}^*(f)(x)-\mathcal{G}_B|
&\lesssim \frac{1}{\mu(\kappa B)} \int_{B} |g_{\lambda,\mu}^*(f)(x)- g_{\lambda,\mu}^*(f \mathbf{1}_{\Rn \setminus B(x_0,\kappa r)})(x)| d\mu(x)\\
&\quad + \frac{1}{\mu(\kappa B)} \int_{B} |g_{\lambda,\mu}^*(f \mathbf{1}_{\Rn \setminus B(x_0,\kappa r)})(x)-g_{\lambda,\mu}^*(f \mathbf{1}_{\Rn \setminus B(x_0,\kappa r)})(x_0)| d\mu(x) \\
&\lesssim \frac{1}{\mu(\kappa B)} \int_{B} |g_{\lambda,\mu}^*(f \mathbf{1}_{B(x_0,\kappa r)})(x)| d\mu(x) + 1\\
&\leq \frac{\mu(B)^{1/2}}{\mu(\kappa B)} ||f \mathbf{1}_{B(x_0,\kappa r)}||_{L^2(\mu)} + 1
\lesssim 1.
\end{align*}
\subsection{Verifying the second condition.}
Next, let us verify the estimate
\begin{equation}\label{G-B}
|\mathcal{G}_{B'} - \mathcal{G}_B| \lesssim 1 + \int_{\kappa B' \setminus B} \frac{d\mu(z)}{|z-x_B|^m},
\end{equation}
whenever $B \subset B'$ are two balls in $\Rn$. Let $x_B$ be the center of $B$.

It follows from Lemma $\ref{key lemma}$ that
\begin{align*}
|\mathcal{G}_{B'} - \mathcal{G}_B|
&\leq |g_{\lambda,\mu}^*(f \mathbf{1}_{\Rn \setminus \kappa B'})(x_{B'})-g_{\lambda,\mu}^*(f \mathbf{1}_{\Rn \setminus \kappa B'})(x_B)| \\
&\quad + |g_{\lambda,\mu}^*(f \mathbf{1}_{\Rn \setminus \kappa B'})(x_{B})-g_{\lambda,\mu}^*(f \mathbf{1}_{\Rn \setminus \kappa B})(x_B)| \\
&\lesssim 1 + |g_{\lambda,\mu}^*(f \mathbf{1}_{\Rn \setminus \kappa B'})(x_{B})-g_{\lambda,\mu}^*(f \mathbf{1}_{\Rn \setminus \kappa B})(x_B)|.
\end{align*}
We choose an increasing sequence of balls $B_0 \subset B_1 \subset \ldots \subset B_K$ satisfying
\begin{enumerate}
\item [(1)] $B_0 = B, B_K = B'$;
\item [(2)] $r(B_j) \simeq r(B_{j+1})$ and $|z-x_B| \simeq r(B_j)$ for $z \in \kappa B_{j+1} \subset \kappa B_j$;
\item [(3)] If $y \in 3B_j$ and $z \in \Rn \setminus \kappa B_j$, then $|y-z| \simeq |z-x_B|$.
\end{enumerate}
Thus, we have
\begin{align*}
&|g_{\lambda,\mu}^*(f \mathbf{1}_{\Rn \setminus \kappa B'})(x_{B})-g_{\lambda,\mu}^*(f \mathbf{1}_{\Rn \setminus \kappa B})(x_B)| \\
&\leq \sum_{j=0}^{K-1} |g_{\lambda,\mu}^*(f \mathbf{1}_{\Rn \setminus \kappa B_{j+1}})(x_{B})-g_{\lambda,\mu}^*(f \mathbf{1}_{\Rn \setminus \kappa B_j})(x_B)| \\
&\leq \sum_{j=0}^{K-1} g_{\lambda,\mu}^*(f \mathbf{1}_{\kappa B_{j+1} \setminus \kappa B_j})(x_{B})
\end{align*}
Consequently, the inequality $(\ref{G-B})$ will be yielded once one obtains
\begin{equation}\label{k-Bk+1}
g_{\lambda,\mu}^*(f \mathbf{1}_{\kappa B_{k+1}\setminus \kappa B_k})(x_0) \lesssim \int_{\kappa B_{k+1}\setminus \kappa B_k}\frac{d\mu(z)}{|y-x_B|^m}.
\end{equation}
We dominate
$$g_{\lambda,\mu}^*(f \mathbf{1}_{\kappa B_{k+1}\setminus \kappa B_k})(x_0) \leq \Lambda_1 + \Lambda_2,$$
where
$$ \Lambda_1 = \bigg(\int_{0}^{r(B_k)}\int_{\Rn} \Big(\frac{t}{t+|x_B-y|}\Big)^{m\lambda}
|\theta_t^{\mu} (f \mathbf{1}_{\kappa B_{k+1}\setminus \kappa B_k})(y)|^2 \frac{d\mu(y) dt}{t^{m+1}}\bigg)^{1/2} $$
and
$$ \Lambda_2 = \bigg(\int_{r(B_k)}^{\infty}\int_{\Rn} \Big(\frac{t}{t+|x_B-y|}\Big)^{m\lambda}
|\theta_t^{\mu} (f \mathbf{1}_{\kappa B_{k+1}\setminus \kappa B_k})(y)|^2 \frac{d\mu(y) dt}{t^{m+1}}\bigg)^{1/2}.$$
Proceeding as we did for $\mathcal{A}_1$, we get
$$ \bigg(\int_{\Rn} \Big(\frac{t}{t+|x_B-y|}\Big)^{m\lambda}
|\theta_t^{\mu} (f \mathbf{1}_{\kappa B_{k+1}\setminus \kappa B_k})(y)|^2 \frac{d\mu(y)}{t^m}\bigg)^{1/2}
\lesssim t^\alpha \int_{\kappa B_{k+1}\setminus \kappa B_k}\frac{d\mu(z)}{|z-x_B|^{m+\alpha}}.$$
Thus, this implies that
\begin{align*}
\Lambda_1 &\lesssim \bigg(\int_{0}^{r(B_k)} t^{2 \alpha - 1} dt\bigg)^{1/2} \int_{\kappa B_{k+1}\setminus \kappa B_k}\frac{d\mu(z)}{|z-x_B|^{m+\alpha}}\\
&\simeq \int_{\kappa B_{k+1}\setminus \kappa B_k}\frac{r(B_k)^\alpha d\mu(z)}{|z-x_B|^{m+\alpha}}
\simeq \int_{\kappa B_{k+1}\setminus \kappa B_k}\frac{d\mu(z)}{|z-x_B|^m}.
\end{align*}
Finally, we bound $I_2$ as follows. The size condition gives that
$$ |\theta_t^{\mu} (f \mathbf{1}_{\kappa B_{k+1}\setminus \kappa B_k})(y)| \lesssim t^{-m} \mu(\kappa B_{k+1}\setminus \kappa B_k).$$
Therefore,
\begin{align*}
\Lambda_2 &\lesssim \mu(\kappa B_{k+1}\setminus \kappa B_k) \bigg(\int_{r(B_k)}^{\infty} \frac{1}{t^{2m+1}}
\int_{\Rn} \Big(\frac{t}{t+|x_B-y|}\Big)^{m\lambda} \frac{d\mu(y)}{t^m} dt\bigg)^{1/2} \\
&\lesssim \frac{\mu(\kappa B_{k+1}\setminus \kappa B_k)}{r(B_k)^m}
\simeq \int_{\kappa B_{k+1}\setminus \kappa B_k}\frac{d\mu(z)}{|z-x_B|^m}.
\end{align*}
This shows $(\ref{k-Bk+1})$.
\qed

\end{document}